\documentclass[11pt]{ip-journal}
\usepackage{tikz}

\usepackage[includefoot]{geometry}
\usepackage{color}
\usepackage{amssymb}
\usepackage{amsmath}
\usepackage{amsthm} 
\usepackage{mathrsfs}                                                            
\usepackage[all]{xy}
\usepackage{graphicx}
\usepackage{wrapfig} 

\definecolor{NoteColor}{rgb}{1,0,0}

\newcommand{\Hom}{\operatorname{Hom}}
\newcommand{\SL}{\operatorname{SL}}
\newcommand{\PSL}{\operatorname{PSL}}
\newcommand{\OSp}{\operatorname{OSp}}
\newcommand{\SO}{\operatorname{SO}}

\newtheorem{theorem}{\rm\bf Theorem}[section]
\newtheorem{proposition}[theorem]{\rm\bf Proposition}
\newtheorem{lemma}[theorem]{\rm\bf Lemma}
\newtheorem{corollary}[theorem]{\rm\bf Corollary}
\newtheorem*{theorem 1}{\rm\bf Proposition 1}
\newtheorem*{theorem 2}{\rm\bf Proposition 2}

\theoremstyle{definition}
\newtheorem{definition}[theorem]{\rm\bf Definition}

\theoremstyle{remark}
\newtheorem{remark}[theorem]{\rm\bf Remark}

\begin{document}

\title{Super McShane identity}
\author{Yi Huang}
\address {Yi Huang, \newline
Yau Mathematical Sciences Center\newline
Tsinghua University, Beijing, China\newline 
yihuangmath@tsinghua.edu.cn}

\author{Robert C. Penner}
\address {Robert C. Penner,\newline
Institut des Hautes \'{E}tudes Scientifiques\newline
35 route des Chartres\newline
Le Bois Marie\newline
91440 Bures-sur-Yvette, France;\newline
Mathematics Department,
UCLA\newline
Los Angeles, CA 90095 USA
\newline
rpenner@ihes.fr}

\author{Anton M. Zeitlin}
\address {Anton M. Zeitlin,\newline
 Department of Mathematics\newline
Louisiana State University\newline
Baton Rouge LA 70803 USA; \newline
IPME RAS, St. Petersburg\newline
zeitlin@lsu.edu\newline
http://math.lsu.edu/$\sim$zeitlin}




\begin{abstract}
The authors derive a McShane identity for once-pun\-ctured super
tori.  Relying upon earlier work  on super Teichm\"uller theory by 
the last two-named authors, they further develop the supergeometry of these surfaces and establish the asymptotic growth rate of their length spectra.
\end{abstract}

\maketitle

\section*{Introduction}

Greg McShane proved in his thesis \cite{greg} the remarkable fact that for any hyperbolic structure (complete metric of constant Gaussian curvature -1) on a once-punctured torus $F$, we have
\[
\frac{1}{2}=\sum_\gamma \frac{1}{e^{\ell_\gamma}+1}\, ,
\]
where the sum is over all the simple closed geodesics $\gamma$ in $F$ and $\ell_\gamma$ is the hyperbolic length of $\gamma$. Maryam Mirzakhani generalized this in her thesis \cite{maryam} to the setting of hyperbolic structures on bordered hyperbolic surfaces and famously employed this to derive an effective recursion for the computation of Weil-Petersson volumes of Riemann moduli spaces \cite{Msim}, prove Witten's conjecture in a novel manner \cite{Mwei} and to establish the asymptotic growth rates of simple closed geodesics on any finite-type hyperbolic surface \cite{Mgro}.\medskip

Relevant to our discussion here, consider the moduli space of flat $G$-connections \cite{morgan} on a surface $F$ of negative Euler characteristic, or equivalently the space of conjugacy classes of representations from $\pi_1(F)$ into a Lie group $G$, where one imposes certain conditions on the holonomy around the punctures, if any. These conditions help to control the singularities which arise on the moduli space obtained from taking the quotient of the representation variety by the conjugation action of G. In the special case where $G=\PSL_2({\mathbb R})$, the maximal dimensional component of this moduli space is the classical Teichm\"uller space of $F$ \cite{goldman}. More generally, for $G$ a split real simple Lie group, certain components of these moduli spaces provide the higher Teichm\"uller spaces (see, e.g., \cite{wienhard} for an overview).  In this context, McShane's identity is the statement that the function given by $\sum_\gamma (e^{\ell_\gamma}+1)^{-1}$ is constant one half on Teichm\"uller space.  A wealth of generalized McShane-type identities have followed (see \cite{britan} for a survey), and recent work \cite{yi+zhe} has provided an analogous identity for positive $G=\SL_3({\mathbb R})$ representations with unipotent boundary holonomy, as well as identities for positive $G=\SL_N({\mathbb R})$ representations with loxodromic boundary holonomy.\medskip

In classical {\sl bosonic} mathematics, the product of real and/or complex numbers commute. Parallel, or perhaps perpendicular, to this familiar setting, lies a world wherein certain further {\sl fermionic} variables anti-commute. The mathematical formalism rests on Grassmann algebras (see Section~\ref{sec:grassmann}) with their grading into even and odd variables respectively corresponding to bosonic and fermionic quantities. This extension is sometimes called {\sl super mathematics}, and derives from the fundamental physical principle  that matter is composed of fermions with the forces mediating them comprised of bosons or, in the language of quantum field theory, wave functions respectively satisfying Fermi-Dirac or Bose-Einstein statistics.  This distinction is basic to contemporary physics and provides a seductive research frontier for mathematics.  In a companion paper \cite{hyper}, from which the current paper is independent, we provide for the first time a detailed mathematical discussion of super hyperbolic geometry.\medskip

Unifying these several threads, one might naturally extend to the case of real split simple Lie supergroups, whose Lie superalgebras were classified by Victor Kac \cite{viktor} in the 1980s.  The simplest such supergroup is the orthosymplectic group $G=\OSp(1|2)$ (see \S\ref{sec:superdec}) which is, morally speaking,  the smallest Lie supergroup containing $\SL_2({\mathbb R})$.  In particular, a representation $\bar{\rho}$ of $\pi_1(F)$ into $\OSp(1|2)$ restricts to a representation
$\rho$ into $\SL_2({\mathbb R})$, and if $\rho$ covers a Fuchsian representation into $\PSL_2({\mathbb R})$, then 
$\rho$  (and hence $\bar\rho$) provides a spin structure on the underlying surface according to
Sergey Natanzon \cite{sergey}.\medskip

In previous work \cite{IPZ1, IPZ2, PZ}, a satisfactory description and parametri\-zation was given of the super Teichm\"uller space for $G=\OSp(1|2)$ pertinent here as well as the higher super Teichm\"uller space for $G=\OSp(2|2)$.  Using these coordinates and mimicking a proof of McShane's identity due to Brian Bowditch \cite{Bpro, Bmar}, we prove the super McShane identity, to wit
\begin{align}
\frac{1}{2}
=\sum_\gamma
\left(
\frac{1}{e^{\ell_\gamma}+1}
+\frac{W_\gamma}{4}~\frac{\sinh\frac{\ell_\gamma}{2}}
{\cosh^2\frac{\ell_\gamma}{2}}
\right),
\end{align}
where $\ell_\gamma$ denotes the super length of $\gamma$, which ranges over the simple closed super geodesics in $F$, and $W_\gamma$ is a 
bosonic product of the two fermionic coordinates on super Teichm\"uller space with order determined by the underlying spin structure on $F$.\medskip

Let us make special mention of the paper \cite{stawit} by Douglas Stanford and Edward Witten, which provides a clear and explicit roadmap for the theory of super McShane identities and Mirzakhani-type super volume recursions for Riemann super moduli spaces. The heuristically \cite{wit} derived formulae they obtain is internally consistent with their matrix model predictions, and is strong support that their formula is correct. In the present paper, we rigorously establish the McShane identity for once-punctured supertori. This is, to the best of our knowledge, the first successful instance of nontrivial superanalysis being conducted on superalgebraic structures to produce a formula with clear supergeometric interpretation. As a consequence, we are rewarded by results on the asymptotic growth rates on super lengths as well as a general bosonic versus fermionic comparison theorem for this simple length spectrum.\medskip

$\S$\ref{sec:dec} is a review of classical decorated Teichm\"{u}ller theory, with particular emphasis on the case when $F$ is the once-punctured torus. $\S$\ref{sec:superdec} introduces basic notions in super arithmetic, which leads into an overview of recent work \cite{IPZ1,IPZ2, PZ} regarding decorated super-Teichm\''{u}ller theory, again with a focus on the case when $F$ is the once-punctured torus. $\S$\ref{sec:supergeo} covers supergeodesics, their super lengths and  hyperbolic elements of the super Fuchsian group, again emphasizing the once-punctured super torus. $\S$\ref{sec:superlamb} introduces the notion of super semi-perimeter, an important invariant of a decorated once-punctured super torus, and explores recursive structure among the super $\lambda$-lengths critical to our proof. $\S$\ref{sec:bodysoul} covers preparatory ingredients for the main proof, including super-Markoff triples and the combinatorial dynamics it imposes on the dual complex of the Farey triangulation. The keystone of this section is the ``body-soul" comparison theorem, which controls the super data in terms of classical. $\S$\ref{sec:proof} is devoted to the proof of the super McShane identity.

\medskip

\subsection*{Acknowledgements} The authors thank Ivan Ip and Edward Witten for valuable input. We also wish to thank the Mathematisches Forschungsinstitut Oberwolfach and the Institut des Hautes \'{E}tudes Scientifiques for hospitality, the former during the inception of this work and latter for the completion. A.M.Z. is partially supported by Simons Collaboration Grant, Award ID: 578501. 
Y.H. was partially supported by the China Postdoctoral Science Foundation grant number 2017T100058.




\medskip

\section{Decorated Teichm\"uller space}
\label{sec:dec}
Fix an oriented surface $F=F_g^s$ of genus $g$ with $s\geq1$ punctures and negative Euler characteristic $2-2g-s<0$. 
We refer to the isotopy class of any homotopically non-trivial simple arc connecting punctures
in $F$ as an \emph{ideal arc}. 
An \emph{ideal triangulation} of $F$ is a maximal collection of ideal arcs with pairwise disjoint representatives.
In the presence of a hyperbolic metric on $F$, we identify each ideal arc with its unique geodesic representative for expositional simplicity. 
The \emph{(pure) mapping class group} $MC(F)$ is defined as the group of isotopy classes of orientation-preserving 
homeomorphisms of $F$ which preserve each puncture.

\medskip

\subsection{Teichm\"uller space}

The \emph{Teichm\"uller space} $T(F)$ of $F$ is the space of $\PSL_2(\mathbb{R})$ conjugacy classes of  \emph{ (type-preserving)} \emph{Fuchsian representations}, that is
\[
T(F)~=~\Hom'(\pi_1(F),G)/G,\quad\text{ for }G=\PSL_2(\mathbb{R}),
\] 
where $\Hom'$ denotes the space of discrete and faithful representations $\rho: \pi_1(F)\rightarrow G$ which map
 puncture-parallel loops in $F$ to parabolic elements of $G$. Each such conjugacy class $[\rho]$  of representation $\rho$
 determines a hyperbolic structure on $F$ via $F=\mathbb{D}/\rho(\pi_1(F))$.

\medskip

\subsection{Decorated Teichm\"uller space}

A basic family of invariants for us are the \emph{horocycles} in $F$, namely, immersed closed curves 
of constant geodesic curvature $1$. Equivalently, horocycles are curves which lift to a horocycle in the universal cover 
$\mathbb{D}$ tangent to $\partial\mathbb{D}$ at a fixed point of a parabolic element of $\rho(\pi_1(F))$.  The hyperbolic
length of a horocycle in $F$ about the corresponding puncture determines it uniquely, and horocycles in $F$ of
length at most unity are embedded.

\medskip

The \emph{decorated Teichm\"uller space} $\tilde{T}(F)$ is the fiber bundle over $T(F)$ where the fiber over each point $[\rho]\in T(F)$ is the
space comprised of all possible specifications of $s$ (potentially intersecting and not necessarily embedded) horocycles, one for each puncture. 
We refer to each such collection of horocycles as a \emph{decoration}, and note that the hyperbolic lengths of these $s$ horocycles parameterize each fiber. 
In particular, the decorated Teichm\"uller space $\tilde{T}(F)$ is a principal $\mathbb{R}_{>0}^s$-bundle. 
We refer to each of these $s$ horocycle length functions parameterizing the $\mathbb{R}_{>0}^s$ fiber as a \emph{perimeter}
and define the action of the mapping class group $MC(F)$ to act trivially on perimeters.

\subsection{$\lambda$-length coordinates}

Given a pair $h,h'\subset \mathbb{D}$ of horocycles, we define the \emph{$\lambda$-length}
$\lambda(h,h')$ as the exponential of one half the signed hyperbolic distance between $h,h'$, 
with positive sign taken if and only if $h$ and $h'$ are disjoint. This construction evidently defines invariants of
ideal arcs on decorated surfaces. In fact, $\lambda$-lengths define global coordinates on the decorated Teichm\"uller space $\tilde{T}(F)$:

\begin{theorem}[{\cite{Pdec,Pbook}}] 
\label{thm:coords}
Fix an ideal triangulation $\Delta$ of $F=F_g^s$, where
$2-2g-s<0$.  Then the assignment 
\[
\tilde{T}(F)\to\mathbb{R}_{>0}^\Delta\cong\mathbb{R}_{>0}^{6g-6+3s}
\]
of $\lambda$-lengths is a real-analytic diffeomorphism onto $\mathbb{R}_{>0}^\Delta$. 
\end{theorem}

To clarify this, we assign a $\lambda$-length coordinate to each
edge of $\Delta$. There are $6g-6+3s$ edges in total, yielding a global chart 
$\mathbb{R}_{>0}^{6g-6+3s}$ on $\tilde{T}(F)$.
We hereafter conflate the name for an edge with its $\lambda$-length for convenience.

\subsection{$h$-lengths vs. perimeters}

Closely related parameters on $\tilde{T}(F)$ that are also central in our work are the \emph{$h$-lengths}, one assigned to
each pair $(v,t)$ in $\Delta$, where $v$ is one of the three ideal vertices of the decorated triangle $t$ complementary to $\Delta$, defined to be the hyperbolic length of the horocyclic segment in the decoration between the ideal edges of $t$ incident on $v$. Direct calculation shows that if $a,b,c$ are the consecutive edges of a decorated triangle $t$, then the respective $h$-lengths of the opposite decorated vertices are given by
\[
\alpha=\frac{a}{bc},\quad\beta=\frac{b}{ac},\quad\gamma=\frac{c}{ab}.
\]
The above formulae allow us to express explicitly the perimeter for any given horocycle
decorating a punctured surface as the sum of the $h$-lengths along the specified horocycle.
In particular, perimeters are rational functions in $\lambda$-lengths with positive integer coefficients. 


\subsection{Ideal cell decomposition of $\tilde{T}(F)$}

As a generalization of an ideal triangulation, an \emph{ideal cell decomposition} $\Delta$ of $F$ is a collection of pairwise disjoint ideal arcs, no two of which are isotopic, so that each component of $F-\cup\Delta$ is an ideal polygon. Equivalently, an ideal cell decomposition is a subset of an ideal triangulation
so that each complementary region is simply connected.

\begin{theorem}[{\cite{Pdec,Pbook}}] \label{thm:icd} For any $F=F_g^s$ of negative Euler characteristic, there is a $MC(F)$-invariant ideal cell decomposition of $\tilde{T}(F)$ isomorphic to the subset of the geometric realization of the arc complex consisting of cells labeled by ideal cell decompositions of $F$.
\end{theorem}

The arena for proving both Theorems \ref{thm:coords} and \ref{thm:icd}, as well as a paradigm for the
super analog of the former, is the Minkowski 3-space
$\mathbb{R}^{2,1}$, that is, $\mathbb{R}^3$ endowed with the standard pairing 
\[
\langle(x,y,z),(x',y',z')\rangle~~=~~-xx'-yy'+zz'
\]
of type $(-,-,+)$.  Equivalently one can use the coordinates $x_1=z-x$, $x_2=z+x$ for which
\[
\langle(x_1,x_2,y),(x_1'x_2',y')\rangle = \tfrac{1}{2}(x_1x_2'+x_2x_1')-yy', 
\]
and $\mathbb{R}^{2,1}\approx\mathbb{R}^3$ is naturally
coordinatized by 
\[
A=A(x,y,z)
=\begin{pmatrix}z-x&y\\ y&z+x\\\end{pmatrix}
=\begin{pmatrix}x_1&y\\ y&x_2\\\end{pmatrix}
\]
with $\SO(2,1)$ furthermore acting naturally on $\mathbb{R}^{2,1}$ by the adjoint
\[
g:A\mapsto g^{\mathrm{t}} Ag,~{\rm for}~g\in \SO(2,1).
\]

\medskip

The upper sheet 
\[
\mathbb{H}=\left\{ u=(x,y,z)\in \mathbb{R}^{2,1}: ~\langle u,u\rangle=1\text{ and }z>0\right\}
\] 
of the hyperboloid 
inherits from $\mathbb{R}^{2,1}$ the structure of the hyperbolic plane where the distance $d$ between
$u,v\in\mathbb{H}$ is given by $\cosh(d)=\langle u,v\rangle$.  The
(open) positive light-cone $L^+$ is
\[
L^+=\left\{ u=(x,y,z)\in \mathbb{R}^{2,1}: ~\langle u,u\rangle=0\text{ and }z>0\right\},
\] 
and affine duality
\[
u\leftrightarrow h(u)=
\left\{ v\in{\mathbb H}:~\langle u,v\rangle~=-2^{-\frac{1}{2}}\right\}
\]
establishes a homeomorphism between $u\in L^+$ and the space of all horocycles $h(u)$ in ${\mathbb H}\approx\mathbb{D}$
in the Hausdorff topology.  This identification
is geometrically natural in the sense that 
\[
\lambda(h(u),h(v))=\sqrt{\langle u,v\rangle}\,.
\]  
In order to prove Theorem \ref{thm:coords}, starting from an arbitrary point $[\tilde{\rho}]\in\tilde{T}(F)$, one recursively employs the $\lambda$-lengths to simultaneously produce

\begin{itemize}
\item
a representation $\pi_1(F)\xrightarrow{\rho}\PSL_2(\mathbb{R})$ and 
\item
a $\pi_1(F)$-equivariant lift 
\[
\ell_\infty:\tilde{\Delta}_\infty \to L^+
\]
from the set of ideal vertices $\tilde{\Delta}_\infty\subset \mathbb{S}^1$ of the pre-image of $\Delta$ in the universal cover
to $L^+$.
\end{itemize}
\medskip

The map $\ell_\infty$ is uniquely determined by the $\lambda$-lengths up to the overall action of $\PSL_2(\mathbb{R})\approx \SO_+(2,1)$,
the identity component of the group $\SO(2,1)$ of Minkowski isometries.  

\medskip

For Theorem \ref{thm:icd}, one takes the closed convex hull of those points in $L^+$ corresponding via affine duality
to the decoration, and the extreme edges of this $\rho(\pi_1(F))$-invariant
convex body project to a corresponding ideal cell decomposition $\Delta([\tilde{\rho}])$ of $F$, for $[\tilde{\rho}]\in \tilde{T}(F)$
covering $[\rho]\in T(F)$. For any fixed ideal cell decomposition $\Delta$, define
\[
\mathcal{C}(\Delta)
=\left\{[\tilde{\rho}]\in \tilde{T}(F):\Delta([\tilde{\rho}])\subsetneq\Delta\right\} .
\]
As $\Delta$ varies over all ideal cell decompositions of $F$, these $\mathcal{C}(\Delta)\subset\tilde{T}(F)$ give a $MC(F)$-invariant decomposition of $\tilde{T}(F)$, and the difficult step is showing that
these putative cells are indeed cells.  In fact, they projectivize to open simplices of dimension $6g-7+3s$ and glue together in the natural way to form an ideal cell decomposition of $\tilde{T}(F)/\mathbb{R}_{>0}$. For $s=1$, this is an ideal cell-decomposition of the Teichm\"uller space $T(F)=\tilde{T}(F)/\mathbb{R}_{>0}$ itself.

\subsection{Flips and the Ptolemy groupoid}

The basic combinatorial operation to produce a new ideal triangulation from an initial triangulation $\Delta$ is called a \emph{flip}. Flips are indexed
by pairs $(e,\Delta)$ consisting of an edge $e$ in the triangulation $\Delta$ which borders two distinct triangles in $F-\Delta$ and are defined as follows: the flip $\Phi_e:=\Phi_{(e,\Delta)}$ removes $e$ from $\Delta$ and replaces it by the other diagonal $f$ of the ideal quadrilateral complementary to $F-(\Delta-\{ e\})$. If $a,b,e$ and $c,d,e$ are the consecutive edges of two such
triangles, then one can compute without difficulty that the {\it Ptolemy relation}
\[
ef=ac+bd
\]
holds among the $\lambda$-lengths.

\medskip

The \emph{Ptolemy groupoid} $Pt(F)$ is the
groupoid whose objects are ideal triangulations of $F$ and whose morphisms
are given by compatible compositions of flips.  According to Theorem~\ref{thm:icd}, this is exactly
the fundamental path groupoid of the 1-skeleton of the dual of the ideal cell decomposition of $\tilde{T}(F)$.
The next result follows immediately from the previous theorem
together with considerations of general position.

\begin{corollary}
For any $F=F_g^s$ of negative Euler characteristic, the Ptolemy groupiod $Pt(F)$ of $F$ is connected: finite compositions of flips act transitively on
ideal triangulations of $F$.  Moreover, a complete set of relations is given by: 

\leftskip=3ex

\medskip

{\rm i)} if $e,f$ are diagonals
of the quadrilateral complementary to $F-(\Delta-\{e\})$, then $\Phi_e\Phi_f=1$; 

\medskip

{\rm ii)} if $e,f\in\Delta$
and their nearby quadrilaterals are disjoint from one another, then $\Phi_e$ and $\Phi_f$ commute; 

\medskip

{\rm iii)} if $e,f\in\Delta$ are edges of a common triangle which are interior to a pentagon complementary to $F-(\Delta-\{e,f\})$,
then the composition of five non-trivial consecutive flips along edges interior to the pentagon yields the identity.  

\medskip

\leftskip=0ex

Moreover, every vertex isotropy group of the Ptolemy groupoid is isomorphic to the mapping class group $MC(F)$.
 \end{corollary}
 
\subsection{Farey tessellation for $T(F^1_{1})$}
\label{sec:fareytess}
In the special case $s=1$, projectivization of the $\lambda$-length coordinates on decorated Teichm\"uller space give coordinates on the Teichm\"uller space
$T(F)$ itself, and a particularly nice section of the forgetful bundle $\tilde{T}(F)\to T(F)$ is given by setting the sum of all the $h$-lengths
to a constant.\medskip

In the still more special case of $g=s=1$, the Teichm\"{u}ller space $T(F)$ identifies with the
Poincar\'e disk $\mathbb{D}$, and the ideal cell decomposition of Theorem~\ref{thm:icd} specializes to a classical picture from 
algebraic number theory, namely the \emph{Farey tessellation} $\tau_*$ defined as follows (see \cite{Pbook} for example for a discussion of the history and significance of $\tau_*$)
take the ideal triangle $t$ spanned by $\pm 1,\sqrt{-1}\in\partial\mathbb{D}\subset{\mathbb C}$ and consider the
group generated by hyperbolic reflections in its sides.  The orbit of the three edges in the frontier of $t$ under this group of reflections
constitutes the Farey tessellation.  

\medskip

The index two normal subgroup of the reflection group consisting of an even number of reflections is the discrete group $\PSL_2(\mathbb{Z})$, called the \emph{modular group}. In particular, the modular group $\PSL_2(\mathbb{Z})$ acts on the Farey tessellation $\tau_*$, and even acts simply transitively on the collection of oriented edges of $\tau_*$.

\medskip

The modular group plays a second role for $F=F_1^1$ since it is the mapping class group
$MC(F)\approx \PSL_2(\mathbb{Z})$.  The Farey tessellation also plays a second role since
any finite-index subgroup $\Gamma<\PSL_2(\mathbb{Z})$ must also leave the Farey tessellation invariant and descend
to an ideal triangulation of the corresponding cover of $F$.  These are precisely the ``punctured arithmetic
surfaces'' characterized by demanding integral $\lambda$-lengths on any ideal triangulation; more explicitly still,
there is a decoration of $\tau_*$ all of whose $\lambda$-lengths are unity, and one checks that the $\lambda$-length of {\sl any}
pair of corresponding horocycles is integral.  We see in $\S$\ref{sec:dualcomplex} that the Farey tesselation plays yet a third role as the na\"{i}ve dual to the so-called curve complex of $F_1^1$.

\subsection{Overview of the $T(F^1_{1})$ theory}

The surface $F=F^1_1$ enjoys special combinatorial properties:\medskip

{Firstly}, given an ideal arc $a$ in $F$, there is a unique isotopy class of simple closed curve $\gamma_a$ in its complement. Indeed, one can see 
that this is manifestly the case for a specific ideal arc and then use the fact that $MC(F)$ acts transitively on ideal arcs.  The converse is true by essentially the same argument, and yields a natural bijection between ideal arcs and simple closed curves.\medskip

{Secondly}, there is a single ideal triangulation of $F$ up to homeomorphism since $\PSL_2(\mathbb{Z})$ acts transitively on the triangles complementary to $\tau_*$, so $MC(F)=Pt(F)$.\medskip

{Thirdly}, if cutting along $\{a,b\}$ decomposes $F$ into an ideal quadrilateral, then there are exactly two ideal arcs $c,c'$ so that $\{a,b,c\}$ and $\{a,b,c'\}$ form ideal triangulations of $F$, and $\{ c,c'\}=\{T_ab,T_a^{-1}b\}$, where $T_a$ is the (right) Dehn twist along $\gamma_a$ and acts on corresponding $\lambda$-lengths by 
\[
(a,b,c)\mapsto \left(a,\tfrac{a^2+b^2}{c},b\right).
\]

{Fourthly}, fix any ideal triangulation and consider the $MC(F)=Pt(F)$ orbit of the point with $\lambda$-lengths $(1,1,1)$.
The \emph{Markoff equation} 
\begin{align}
a^2+b^2+c^2=3abc
\end{align}
arises directly from the Ptolemy equation by solving the former as a quadratic in $c$,
and its solutions correspond to arithmetic 
once-punctured tori.  The constant 3 in the equation is flexible in the sense that projective re-scaling 
alters this constant in the inhomogeneous equation.  Indeed this scale is tantamount to a choice of
section of $\tilde{T}(F)\to T(F)$ insofar as the horocycle on $F$ corresponding to a triple $(a,b,c)$
has hyperbolic length given by the sum of all $h$-lengths, namely,
\[
2\left(\tfrac{a}{bc}+\tfrac{b}{ac}+\tfrac{c}{ab}\right)=2\tfrac{a^2+b^2+c^2}{abc}.
\]

\medskip

\noindent Summarizing, we have

\medskip

\begin{corollary}\label{cor:torus}
For $F=F_1^1$, the ideal cell decomposition of $T(F)$ in Theorem~\ref{thm:icd}  
 is the Farey tessellation $\tau_*$,
and in this case we have $MC(F)\approx \PSL_2(\mathbb{Z})$ leaving $\tau_*$ invariant.  
For any ideal triangulation, the flip-orbit of the point $(1,1,1)$ in $\lambda$-lengths is given by
the {\rm Markoff triples}, i.e., solutions to the Diophantine equation $a^2+b^2+c^2=3abc$,
and the effect of the Dehn twist along $\gamma_a$ on $\lambda$-length coordinates is given by
$(a,b,c)\mapsto (a,{{a^2+b^2}\over c},b)$.  Moreover, sections of $\tilde{T}(F)
\to T(F)
$ are given in coordinates
by level sets of ${{a^2+b^2+c^2}\over{abc}}$.
\end{corollary}

\section{Decorated Super Teichm\"uller space}
\label{sec:superdec} 
\subsection{Elements of superalgebra and the $\OSp(1|2)$ supergroup}
\label{sec:grassmann}
We will be working in a finite dimensional Grassmann algebra $\mathbb{R}_{S[N]}$ with generators $1,\beta_{[1]},\ldots,\beta_{[N]}$ (see, for example, \cite[Chapter~3]{Rbook}). 
It is also useful to introduce notation 
$\beta_{[\lambda]}=\beta_{[\lambda_1]}\dots \beta_{[\lambda_k]}$, where $\lambda=\lambda_1\dots\lambda_k$ is a multi-index, so that $\lambda_1<\lambda_2<\dots<\lambda_k$.
Thus, any $x\in \mathbb{R}_{S[N]}$ can be expanded as 
\[
x=\sum_{\lambda}x_{\lambda}\beta_{[\lambda]},
\]
 where $x_{\lambda}\in \mathbb{R}$.\medskip

Up to a change of basis, we assume without loss of generality that $W=\beta_{[\lambda]}$ (this is non-essential, but simplifies our computations slightly in the proof of Lemma~\ref{thm:technicalcompare}). We define the \emph{body map} (also known as the \emph{augmentation map}) to be the projection $\epsilon:\mathbb{R}_{S[N]}\to\mathbb{R}$ which maps to the real coefficient of $1$, and the \emph{soul map}
\[
s:\mathbb{R}_{S[N]}\to\mathbb{R}_{S[N]},\; x\mapsto x-\epsilon(x)\cdot 1.
\]
The body and the soul of an element of a Grassmann algebra are canonically defined, and independent of the choice of basis for $\mathbb{R}_{S[N]}$. In contrast, we introduce a degree-based refinement of the soul map which is basis dependent, but is nevertheless useful for our analytical purposes.\medskip

Define the \emph{degree $k$ soul map} to be the projection of $s_k:\mathbb{R}_{S[N]}\to\mathbb{R}_{S[N]}$ onto the subset consisting of degree $k$ terms, for instance,
\[
s_2(\sqrt{7}+3\beta_{[1]}+5\beta_{[1,3]}-4\beta_{[3,7]}-19\beta_{[2,3,4,5]})
=5\beta_{[1,3]}-4\beta_{[3,7]}.
\]
Furthermore, we set $s_0(x)=\epsilon(x)\cdot 1$ by convention. The norm on this space is defined as follows. Given $x\in \mathbb{R}_{S[N]}$, 
we can write it as 
\[
||x||=\sum_{\lambda}|x_{\lambda}|.
\]
In addition to the standard norm properties the above definition gives $\mathbb{R}_{S[N]}$ the structure of the Banach algebra, namely, $||XY||\le ||X||\cdot||Y||$.\medskip

The even and odd splitting of the Gassmann algebra is another canonically defined algebraic splitting. This time, we partition $\mathbb{R}_{S[N]}$ into $\Lambda_0\oplus\Lambda_1$, where
$\Lambda_0$ is the subalgebra generated by elements of even degree and $\Lambda_1$ is the subalgebra generated by elements of odd degree, for instance,
\begin{align}
\sqrt{7}+3\beta_{[1]}+5\beta_{[1,3]}-4\beta_{[3,7]}-19\beta_{[2,4,5]}
=&\left(\sqrt{7}+5\beta_{[1,3]}-4\beta_{[3,7]}\right)\nonumber\\
+&\left(3\beta_{[1]}-19\beta_{[2,4,5]}\right)\nonumber.
\end{align}
Whereas odd terms anti-commute with one another, even terms commute with everything in the algebra. Moreover, it is easy to see that the even-even and odd-odd products yield even elements whereas odd-even and even-odd products yields odd terms.

\begin{remark}
We will refer to elements of $\Lambda_0$ as even (or bosonic) and to elements of  $\Lambda_1$ as odd (or fermionic). Another notation convention we adopt is the following: we write that 
\[
u\in\Lambda_0>0\text{ if }\epsilon(u)>0\text{, and }u\in\Lambda_0<0\text{ if }\epsilon(u)<0.
\]
\end{remark}

For the $\mathcal{N}=1$ super Teichm\"uller space, where $\mathcal N$ is half the number of real fermions, the role of the Lie group $\PSL_2(\mathbb{R})$ is supplanted by the
\emph{orthosymplectic group} $G=\OSp(1|2)$, a Lie supergroup which we briefly recall. 
An element $g\in G=\OSp(1|2)$ is described by a
matrix
\[
g=
\begin{pmatrix}a&b&\alpha\\c&d&\beta\\ \gamma&\delta &f\\\end{pmatrix}
\quad\text{ satisfying }\quad
g^\mathrm{st}=Jg^{-1}J^{-1},
\]
where $a,b,c,d,f$ are even and $\alpha,\beta,\gamma,\delta$ are odd variables with 
\[
J=
\begin{pmatrix}
\hskip1.6ex0&1&0\\
 -1&0&0\\ 
 \hskip1.6ex 0&0&1
 \end{pmatrix}
\quad\text{and the \emph{supertranspose}}\quad
g^\mathrm{st}=
\begin{pmatrix}
\hskip1.6exa&\hskip1.6exc&\gamma\\
\hskip1.6exb&\hskip1.6exd&\delta\\
 -\alpha&-\beta &f
\end{pmatrix},
\]
so that the \emph{Berezinian} or \emph{superdeterminant} 
\[
\operatorname{sdet}~g~=~f^{-1}
~\det\left[
\begin{pmatrix}
a&b\\ c&d
\end{pmatrix}
+f^{-1}
\begin{pmatrix}
\alpha\gamma&\alpha\delta\\ \beta\gamma&\beta\delta
\end{pmatrix}
\right]
\]
takes the value unity\footnote{In this formalism, the product in $\OSp(1|2)$ is given by
$\biggl ( \begin{smallmatrix}
a_1&b_1&\alpha_1\\c_1&d_1&\beta_1\\\gamma_1&\delta_1&f_1\\
\end{smallmatrix}\biggr )
\biggl ( \begin{smallmatrix}
a_2&b_2&\alpha_2\\c_2&d_2&\beta_2\\\gamma_2&\delta_2&f_2\\
\end{smallmatrix}\biggr )=
\biggl (\begin{smallmatrix}
a_1a_2+b_1c_2-\alpha_1\gamma_2&a_1b_2+b_1d_2-\alpha_1\delta_2&a_1\alpha_2+b_1\beta_2+\alpha_1 f_2\\
c_1a_2+d_1c_2-\beta_1\gamma_2&c_1b_2+d_1d_2-\beta_1\delta_2&c_1\alpha_2+d_1\beta_2+\beta_1 f_2\\
\gamma_1 a_2+\delta_1 c_2+f_1\gamma_2&\gamma_1 b_2+\delta_1 d_2+f_1 \delta_2&-\gamma_1\alpha_2-\delta_1\beta_2+f_1f_2\\
\end{smallmatrix}\biggr ).
$}.  The assignment of the body of the block $\left(\begin{smallmatrix}a&b\\ c&d\end{smallmatrix}\right)$ of $g$
gives a canonical restriction $\OSp(1|2)\to \SL_2(\mathbb{R})$, so one might think of 
$\OSp(1|2)$ as a 
Grassmann extension of $\SL_2(\mathbb{R})$, rather than $\PSL_2(\mathbb{R})$,
with two odd degrees of freedom, where the apparent four degrees $\alpha,\beta,\gamma,\delta$ are halved by the defining equation. Interested readers may see
\cite{PZ}, for instance, for greater detail.

\subsection{Super Teichm\"{u}ller space}

For $F=F_g^s$ with negative Euler characteristic, the ($\mathcal{N}=1$) \emph{ super Teichm\"uller space} is 
\[
ST (F)=~\Hom'(\pi_1(F),G)/G,\quad\text{ for }G=\OSp(1|2),
\]
where
\begin{itemize}
\item
the restriction $\OSp(1|2) \to \SL_2(\mathbb{R})$ followed by the canonical projection
$\SL_2(\mathbb{R})\to \PSL_2(\mathbb{R})$ is required to be Fuchsian; 
\item
the body of the monodromy about each puncture is parabolic, i.e., has trace~$=\pm2$.
\end{itemize}
In particular, the $\SL_2(\mathbb{R})$-restricted boundary trace is given by either $+2$ or $-2$ and respectively correspond to \emph{Ramond} and \emph{Neveu-Schwarz} punctures. We let $n_R$ and $n_{NS}$ respectively denote the numbers of each type of puncture. It is well-known that $n_R$ is always even, cf.\ \cite{Witten}.

\subsection{Spin structures}
The set of spin structures on $F$ is an affine $H^1(F;\mathbb{Z}/2)$-space and index the connected components of the super Teichm\"uller space $ST(F)$. Given an ideal triangulation $\Delta$ of $F$, a spin structure is naturally described by an equivalence class of orientations on the edges of $\Delta$, where two such orientations are equivalent if they are related by a finite sequence of the following \emph{reversals}: choose a triangle complementary to $\Delta$ and change the orientation of each of its frontier edges.\medskip

The action of the Ptolemy groupoid affects the spin structure, and the effect of a flip is
illustrated in Figure~\ref{fig:spinflip}\footnote{We emphasize that this is how flips act on spin structures in the generic setting whereby the edges $a,b,c,d$ are \emph{all distinct edges}. In the case of the once-punctured torus, which is of the focus of this paper, the opposite edges $a,c$ and $b,d$ respectively agree, and the adjusted flip transformation is given by Figure~\ref{fig:spinfliptorus}.}.

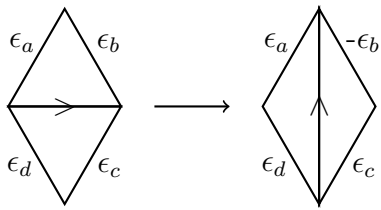
\begin{figure}[h!]

\centering

\begin{tikzpicture}[scale=0.5, baseline,thick]

\draw (0,0)--(3,0)--(60:3)--cycle;

\draw (0,0)--(3,0)--(-60:3)--cycle;

\draw node[above] at (72:1.2){$\epsilon_a$};

\draw node[above] at (23:2.9){$\epsilon_b$};

\draw node[below] at (-22:2.9){$\epsilon_c$};

\draw node[below=-0.1] at (-74:1.1){$\epsilon_d$};

\draw node[above] at (1.5,-0.55){$>$};

\draw node[left] at (0,0) {};

\draw node[above] at (60:3) {};

\draw node[right] at (3,0) {};

\draw node[below] at (-60:3) {};

\end{tikzpicture}
\begin{tikzpicture}[baseline]

\draw[->, thick](0,0)--(1,0);

\node[above]  at (0.5,0) {};

\end{tikzpicture}
\begin{tikzpicture}[scale=0.5, baseline, thick]

\draw (0,0)--(60:3)--(-60:3)--cycle;

\draw (3,0)--(60:3)--(-60:3)--cycle;

\draw node[above] at (72:1.2){$\epsilon_a$};

\draw node[above] at (23:2.9){$\textrm{-}\epsilon_b$};

\draw node[below] at (-22:2.9){$\epsilon_c$};

\draw node[below=-0.1] at (-74:1.1){$\epsilon_d$};

\draw node[left] at (2.0941,0){${\mathbf{\wedge}}$};

\draw node[left] at (0,0) {};

\draw node[above] at (60:3) {};

\draw node[right] at (3,0) {};

\draw node[below] at (-60:3) {};

\end{tikzpicture}
\caption{Flip effect on spin structures}
\label{fig:spinflip}
\end{figure}
The five edges depicted here are taken to be pairwise distinct, and each $\epsilon_x$ denotes an orientation on the ideal arc $x$ with $-\epsilon_x$ denoting orientation reversal of $\epsilon_x$.\medskip

According to \cite{sergey}, a spin structure on a Riemann surface described by a Fuchsian representation $\pi_1(F)\xrightarrow{  \rho  } \PSL_2(\mathbb{R})$  is given by a lift of $\rho$ to $\SL_2(\mathbb{R})$,
consistent with the restriction of a representation as above into $\OSp(1|2)$ inducing a representation into $\SL_2(\mathbb{R})$.
The spin structure thus determines the type R or NS of a puncture.  A complete discussion of spin structures on punctured surfaces and the proof of the combinatorial description above are given in \cite{PZ}.

\subsection{Super Minkowski space}

Here we review facts about super Minkowski space, referring the reader to \cite{PZ} for a further detailed discussion.\medskip

Just as our discussion of the bosonic case relied on Minkowski space $\mathbb{R}^{2,1}$, the current ${\mathcal N}=1$
case employs super Minkowski space $\mathbb{R}^{2,1|2}$, understood as the superspace over some Grassmann algebra $\Lambda$ consisting of vectors $(x_1,x_2,y|\phi,\theta)$, where $x_1,x_2, y\in \Lambda_0$ and $\phi, \theta\in \Lambda_1$, so that  
the pairing between two vectors
$u=(x_1,x_2,y|\phi,\theta)$ and $u'=(x_1',x_2',y'|\phi',\theta')$ is given by
\[
\langle u,u'\rangle
~=~\tfrac{1}{2}(x_1x_2'+x_1'x_2)-yy'+\phi\theta'+\phi'\theta,
\]
and we define the \emph{(super) $\lambda$-length} to be the positive square root (i.e., the square root with positive body) of this inner product for two points in
the \emph{positive light cone} \footnote{The inequality notation $a>0$ or 
$a< 0$  for a given supernumber $a$ mean that such inequalities hold for its body $\epsilon(a)$.}
\[
{\mathcal L}^+
=\left\{
u=(x_1,x_2,y|\phi,\theta)\in{\mathcal L}\subset{\mathbb{R}^{2,1|2}}
:~\langle u,u\rangle~=0~{\rm and}~x_1+x_2>0
\right\},
\]
where ${\mathcal L}$ is the collection of isotropic vectors.  Each $\OSp(1|2)$ orbit in
$\mathcal{L}^+$ contains elements $(1,0,0|0,\theta)$ and $(1,0,0|0,-\theta)$, for some
fermion $\theta$, and $\pm\theta$
uniquely determines the orbit.  We distinguish the subset 
\[
{\mathcal L}_0=\OSp(1|2)\cdot(1,0,0|0,0)\subset\mathcal{L}^+
\]
called the {\it special light cone} which is the super analogue of the bosonic open
positive light cone.  To explain this, one can compute that these are exactly the
vectors in ${\mathcal L}^+$ whose stabilizer in $\OSp(1|2)$ is conjugate in $\OSp(1|2)$
to a unipotent element in analogy to the bosonic situation.  Another more Lie-theoretic
explanation is that ${\mathcal L}_0$ is precisely the intersection between ${\mathcal L}^+$
and the $\OSp(1|2)$ orbit of the highest weight vector.  Just as in the bosonic case, a super horocycle
is the affine dual in super Minkowski space of a point of ${\mathcal L}_0$.\medskip

Any vector in ${\mathbb{R}^{2,1|2}}$ can be written as
\[
A=\begin{pmatrix}x_1&y&\phi\\ y&x_2&\theta\\ -\phi&-\theta&0\\\end{pmatrix}.
\]
Just as in the bosonic case, $\OSp(1|2)$ acts
naturally on ${\mathbb{R}^{2,1|2}}$, ${\mathcal L}$ and ${\mathcal L}_0$ by the adjoint
\[
g:A\to g^{\rm st} A g, ~{\rm for}~g\in \OSp(1|2).
\]

\subsection{Decorated super Teichm\"uller space}

Again just as in the bosonic case, the \emph{decorated super Teichm\"uller space} $S\tilde{T}(F)$ is the space of $\OSp(1|2)$-orbits
of lifts 
\[
\ell_\infty:\tilde\Delta_\infty\to{\mathcal L}_0
\]
which are $\pi_1(F)$-equivariant for some representation $\pi_1(F)\xrightarrow{ \hat\rho}\OSp(1|2)$, such that $[\hat{\rho}]\in ST(F)$.

\begin{theorem}[{\cite{IPZ2,PZ}}]
\label{thm:super} 
For $F=F_g^s$ of negative Euler characteristic, the components of $S\tilde{T}(F)$
are indexed by the spin structures on $F$, and on each component $C$, there are global affine coordinates which assign to an ideal triangulation $\Delta$ of $F$ one even coordinate called a \emph{$\lambda$-length} to each edge of $\Delta$ and one odd coordinate called a \emph{$\mu$-invariant} to each triangle complementary in $F-\Delta$. Moreover, the collection of $\mu$-invariants are taken modulo an overall change of sign, giving a real-analytic homeomorphism from $S\tilde{T}(F)$ onto
\[C\to \mathbb{R}_{>0}^{6g-6+3s|4g-4+2s}/\mathbb{Z}_2.\footnote{The $>0$ notation here means that bodies of all $6g-6+3s$ even coordinates are $>0$ according to the conventions of Section~\ref{sec:grassmann}.  }\]

\medskip

Flips between ideal triangulations act real analytically on their respective $\lambda$-length and $\mu$-invariant coordinates as follows:
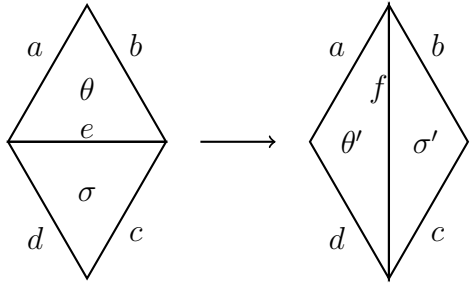
\begin{figure}[h!]

\centering

\begin{tikzpicture}[scale=0.7, baseline, thick]

\draw (0,0)--(3,0)--(60:3)--cycle;

\draw (0,0)--(3,0)--(-60:3)--cycle;

\draw node[above] at (70:1.5){$a$};

\draw node[above] at (30:2.8){$b$};

\draw node[below] at (-30:2.8){$c$};

\draw node[below=-0.1] at (-70:1.5){$d$};

\draw node[above] at (1.5,-0.15){$e$};

\draw node[left] at (0,0) {};

\draw node[above] at (60:3) {};

\draw node[right] at (3,0) {};

\draw node[below] at (-60:3) {};

\draw node at (1.5,1){$\theta$};

\draw node at (1.5,-1){$\sigma$};

\end{tikzpicture}
\begin{tikzpicture}[baseline]

\draw[->, thick](0,0)--(1,0);

\node[above]  at (0.5,0) {};

\end{tikzpicture}
\begin{tikzpicture}[scale=0.7, baseline, thick]

\draw (0,0)--(60:3)--(-60:3)--cycle;

\draw (3,0)--(60:3)--(-60:3)--cycle;

\draw node[above] at (70:1.5){$a$};

\draw node[above] at (30:2.8){$b$};

\draw node[below] at (-30:2.8){$c$};

\draw node[below=-0.1] at (-70:1.5){$d$};

\draw node[left] at (1.7,1){$f$};

\draw node[left] at (0,0) {};

\draw node[above] at (60:3) {};

\draw node[right] at (3,0) {};

\draw node[below] at (-60:3) {};

\draw node at (0.8,0){$\theta'$};

\draw node at (2.2,0){$\sigma'$};

\end{tikzpicture}
\caption{Ptolemy transformation}
\label{fig:ptolemy}
\end{figure}

\begin{align}
ef&=(ac+bd)\biggl( 1+{{\sigma\theta\sqrt{\chi}}\over{1+\chi}}\biggr)\\
\sigma'&={{\sigma-\sqrt{\chi}\theta}\over{\sqrt{1+\chi}}}
\quad\text{ and }\quad
\theta'={{\theta+\sqrt{\chi}\sigma}\over{\sqrt{1+\chi}}}\label{eq:muinvars}
\end{align}
in the notation of Figures~\ref{fig:spinflip} and \ref{fig:ptolemy}, 
where Roman letters are $\lambda$-lengths and Greek are $\mu$-invariants,
subscripted epsilons in Figure ~\ref{fig:spinflip} denote orientations with negation 
signifying reversing the orientation on the interior edge
as indicated and $\chi={{ac}\over{bd}}$ is the {\rm super cross ratio}.
\end{theorem}

\medskip

 On each component of $S\tilde{T}(F)$ are determined the type --- either Ramond or Neveu-Schwarz, of each puncture.  We let $n_{\rm R}$ denote the number of the former and $n_{\rm NS}$ of the latter, so in particular $s=n_{\rm NS}+n_{\rm R}$.  According to \cite{IPZ2}, each Ramond puncture imposes one real odd constraint, and by Theorem \ref{thm:super} the real odd dimension of $S\tilde{T}(F_g^s)$ is
$$4g-4+2s=4g-4+2(n_{\rm NS}+n_{\rm R})=2(2g-2+n_{\rm NS}+n_{\rm R}/2)+n_{\rm R}.$$
Insofar as \cite{Witten} argues from deformation theory that the complex odd dimension of super moduli space is  $2g-2+n_{\rm NS}+n_{\rm R}/2$, we see that the decoration adds to the even variables as usual and to the odd variables one extra real degree of freedom for each Ramond puncture (see \S1 and \S5 of \cite{IPZ2} for more details).\medskip

 The treatment for ${\mathcal N}=2$ in \cite{IPZ2} is entirely analogous (though without the
distinction between Ramond and Neveu-Schwarz punctures) but  is much more involved.

\medskip

Notice in particular that flipping along edge $e$ in Figure~\ref{fig:ptolemy} leaves $\sigma'\theta'=\sigma\theta$ invariant.  Moreover, when $s=1$, the unique puncture is Neveu-Schwarz since Ramond punctures always come in pairs.

\subsection{Once-punctured super tori}

In \cite[Appendix~ B]{IPZ1}, the surface $F=F_1^1$ was considered in detail, and we recall

\begin{proposition}\label{cor:storus} 
For $F=F_1^1$ the unique puncture is NS, and the super decorated Teichm\"uller space $S\tilde{T}(F)$ has four components.  For each such component and each ideal triangulation $\Delta$ of $F$, the component is coordinatized
by three $\lambda$-lengths, one for each of the three edges of $\Delta$, and two $\mu$-invariants, one for each triangle complementary to $\Delta$ in $F$.  The flip on edge $e$ in $\Delta$ acts on $\lambda$-length coordinates $a=c,b=d$ and $e$ as per the notation of Figure~\ref{fig:ptolemy}
by 
$$ef=a^2+b^2+ab\sigma\theta$$ 
and on $\mu$-invariants by
$$\theta'={{b\theta+a\sigma}\over{\sqrt{a^2+b^2}}},~~
\sigma'={{b\sigma-a\theta}\over{\sqrt{a^2+b^2}}}.$$
while the spin structure transformation\footnote{This corrects the last line of Corollary~ B.2 in  \cite[Appendix~B]{IPZ1}.} is described in Figure~\ref{fig:spinfliptorus} (cf. Figure~\ref{fig:spinflip})
\begin{figure}[h!]

\centering

\begin{tikzpicture}[scale=0.5, baseline,thick]

\draw (0,0)--(3,0)--(60:3)--cycle;

\draw (0,0)--(3,0)--(-60:3)--cycle;

\draw node[above] at (72:1.2){$\varepsilon_a$};

\draw node[above] at (23:2.9){$\varepsilon_b$};

\draw node[below] at (-22:2.9){$\varepsilon_a$};

\draw node[below=-0.1] at (-74:1.1){$\varepsilon_b$};

\draw node[above] at (1.5,-0.55){$>$};

\draw node[left] at (0,0) {};

\draw node[above] at (60:3) {};

\draw node[right] at (3,0) {};

\draw node[below] at (-60:3) {};

\draw node at (1.5,1){$\theta$};

\draw node at (1.5,-1){$\sigma$};

\end{tikzpicture}
\begin{tikzpicture}[baseline]

\draw[->, thick](0,0)--(1,0);

\node[above]  at (0.5,0) {};

\end{tikzpicture}
\begin{tikzpicture}[scale=0.5, baseline, thick]

\draw (0,0)--(60:3)--(-60:3)--cycle;

\draw (3,0)--(60:3)--(-60:3)--cycle;

\draw node[above] at (72:1.2){$\varepsilon_a$};

\draw node[above] at (23:2.9){$-\varepsilon_b$};

\draw node[below] at (-22:2.9){$\varepsilon_a$};

\draw node[below=-0.1] at (-74:1.1){$-\varepsilon_b$};

\draw node[left] at (2.0941,0){${\mathbf{\wedge}}$};

\draw node[left] at (0,0) {};

\draw node[above] at (60:3) {};

\draw node[right] at (3,0) {};

\draw node[below] at (-60:3) {};

\draw node at (0.8,0){$\theta'$};

\draw node at (2.2,0){$\sigma'$};

\end{tikzpicture}
\caption{Flip effect on spin structures for $F^1_1$.}
\label{fig:spinfliptorus}
\end{figure}
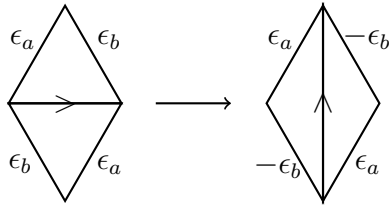
\end{proposition}

\begin{remark}
 For the surface $F^1_1$, there are four spin structures,  one being odd and the other three even. This is detected by the parity of the Arf invariant \cite{arf}, which is  preserved by the mapping class group. 
\end{remark}

\section{Super geodesics}
\label{sec:supergeo}

There is a subtlety regarding geodesics in passing from classical bosonic Riemann surfaces to super Riemann surfaces in that not all geodesics in the upper sheet of the two-sheeted hyperboloid are asymptotic to ${\mathcal L}_0$ and consequently may have no projection to the $(1|1)$-dimensional ideal boundary of the super hyperbolic plane.  We find by explicit calculation, however, that the monodromy along any simple closed super geodesic on a once-punctured super torus can be diagonalized in $\OSp(1|2)$.  Thus, the geodesic lift of the curve to the hyperboloid is therefore indeed asymptotic to ${\mathcal L}_0$, and its monodromy can be conjugated into the subgroup
$\SL_2({\mathbb R})<\OSp(1|2)$.
Moreover, the relationship between trace of monodromy and length of closed curve familiar from hyperbolic surfaces (which we shall recall later) then naturally promotes to a relationship between supertrace and super length. This extension being the main result of this section.\medskip

We first consider the case of arbitrary space-like geodesics in super Minkowski space
\footnote{For a general discussion of geodesics on a supermanifold we refer to the book \cite[\S2.7]{dewitt}. Instead of following the Christoffel symbols approach to geodesics, we use the variational principle based methods which uses Lagrange multiplier technique.}:

\begin{theorem}\label{Thm:geodesic}
Geodesics on the super hyperboloid of two sheets are characterized 
by the equation
\begin{eqnarray}
{\bf x}={\bf u}\cosh t+{\bf v}\sinh t,
\end{eqnarray}
where $\langle{\bf u}, {\bf u}\rangle =1$, $\langle{\bf v}, {\bf v}\rangle =-1$ and $\langle{\bf u}, {\bf v}\rangle =0$. 
The asymptotes of super geodesics are given by the rays in $\mathcal{L}$ determined by the vectors
\begin{eqnarray}
\bf e=\bf u+ \bf v, \quad \bf f=\bf u-\bf v.
\end{eqnarray}
Conversely,  two rays in $\mathcal{L}$ uniquely define a geodesic where ${\bf e}, {\bf f}$ give rise to ${\bf u}, {\bf v}$ in accordance with this formula.   
\end{theorem}

\begin{proof}
To describe geodesics on the upper sheet of the hyperboloid, let us apply the variational principle to minimize the functional 
\[
\int \Big(\sqrt{ |\langle \dot{\bf x}, \dot{\bf x}\rangle|}+\lambda( \langle {\bf x}, {\bf x}\rangle-1)\Big)dt ,
\]
where the dot stands for the derivative with respect to the parameter $t$ along the curve.
The corresponding Euler-Lagrange equations are given by 
\begin{eqnarray}\label{geod}
{\ddot{\bf x}}=2\lambda {\bf x}, \quad \langle {\bf x}, {\bf x}\rangle=1,
\end{eqnarray}
where we may assume that $|\langle \dot{\bf x}, \dot{\bf x}\rangle |=1$ along the geodesic. 
Differentiating the second equation and applying the first equation in (\ref{geod}), we obtain
\begin{eqnarray}
\lambda=-\frac{\langle \dot{\bf x}, \dot{\bf x}\rangle}{2}
\end{eqnarray}

It is left to the reader to check that there are no geodesics with  $\langle \dot{\bf x}, \dot{\bf x}\rangle=1$,
and we focus on the case $\langle \dot{\bf x}, \dot{\bf x}\rangle=-1$ leading to $\lambda=1/2$. 
In this case, the solution to the first equation is given by 
\begin{eqnarray}
{\bf x}={\bf u}\cosh t+{\bf v}\sinh t.
\end{eqnarray}
Imposing the conditions  $-\langle \dot{\bf x}, \dot{\bf x}\rangle=\langle {\bf x}, {\bf x}\rangle=1$, we find that $\bf u, \bf v$ satisfy the conditions of the theorem. 
The remaining part of the assertion follows from the appropriate normalization of light cone vectors.
\end{proof}

\subsection{Generators of the super Fuchsian group}
We next explicitly compute generators of the super Fuchsian group from the coordinates of Corollary~\ref{cor:storus} on $S\tilde T(F_1^1)$.
To this end, choose a fundamental domain as depicted in Figure~\ref{abcd} 
\begin{figure}[h!]
\centering

\begin{tikzpicture}[scale=0.7, baseline, thick]

\draw (0,0)--(3,0)--(60:3)--cycle;

\draw (0,0)--(3,0)--(-60:3)--cycle;

\draw node[above] at (70:1.5){$a$};

\draw node[above] at (30:2.8){$b$};

\draw node[below] at (-30:2.8){$a$};

\draw node[below=-0.1] at (-70:1.5){$b$};

\draw node[above] at (1.5,-0.05){$c$};

\draw node[above] at (1.5,-0.40){$>$};

\draw node[left] at (0,0) {};

\draw node[above] at (60:3) {};

\draw node[right] at (3,0) {};

\draw node[below] at (-60:3) {};

\draw node at (1.5,1){$\theta$};

\draw node at (1.5,-1){$\sigma$};

\draw node[left] at (0,0) {$A$};
\draw node[above] at (60:3) {$B$};
\draw node[right] at (3,0) {$C$};
\draw node[below] at (-60:3) {$D$};

\end{tikzpicture}
\caption{Coordinates for $S\tilde{T}(F_1^1)$}
\label{abcd}
\end{figure}
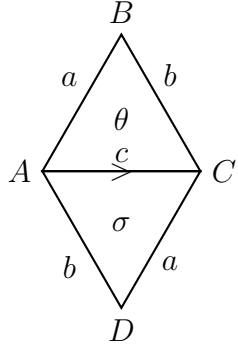
and normalize it as in \cite{PZ}
by applying an element of $\OSp(1|2)$ carrying lifts of the vertices (which lie in ${\mathcal L}_0$ by construction
and which we here identify with the vertices themselves for convenience)
to 
\begin{eqnarray}\label{ABCD}
&&A=u(0,1,0,0,0), ~B=t(1,1,1,\theta, \theta),\nonumber\\
&&C=s(1,0,0,0,0),~ D=(x_1,x_2,-y,\rho, \lambda), 
\end{eqnarray}
where 
\begin{eqnarray}
&&u=\sqrt{2}~\frac{ca}{b}, ~ s=\sqrt{2}~\frac{bc}{a}, ~t=\sqrt{2}~\frac{ab}{c},\nonumber\\
&&x_1=\sqrt{2}~\frac{b^3}{ca}, ~ x_2=\sqrt{2}~\frac{a^3}{cb}, ~ y=\sqrt{2}~\frac{ab}{c},\nonumber \\
&&\lambda=-\sqrt{2}~\frac{a^2}{c}\sigma,~ \rho=\sqrt{2} ~\frac{b^2}{c}\sigma.\nonumber
\end{eqnarray}
Denote generators of the super Fuchsian group by 
$g_a, g_b\in \OSp(1|2)$ respectively mapping ${B,C}$ to ${A,D}$ and ${A,B}$ to ${D,C}$. \medskip

Let us first compute $g_a$ and consider some $U\in \OSp(1|2)$ mapping $C$ to $D$, which is uniquely determined up to precomposition by the stabilizer of $C$. Such an element, stabilizing $C$,  is given by a matrix  
\begin{eqnarray}\label{vmat}
V=V_\beta^{k,q}=\begin{pmatrix}-1&0&0\\ 0&-1&0 \\ 0&0&1\\\end{pmatrix}^k
\begin{pmatrix}1&0&0\\ q&1&\beta \\ \beta &0&1\\\end{pmatrix}.
\end{eqnarray}
 One possible choice for $U$ is given by
\[
U=\begin{pmatrix}\sqrt{\frac{x_1}{s}}&-\sqrt{\frac{x_2}{s}}&\frac{\rho}{\sqrt{x_1s}}\\ \sqrt{\frac{s}{x_2}}&0&0\\ -\frac{\rho}{\sqrt{x_1x_2}}&0&1\\\end{pmatrix},
\]
and the general element mapping $C$ to $D$ therefore has the form $VU$. 
The additional condition that $VU$ maps $B$  to $A$ then completely determines $V$: direct computation shows that 
\[
k=0,\ q=-1-\tfrac{c^2}{a^2}-\tfrac{c}{a}\sigma\theta, ~\beta=\tfrac{c}{a}\sigma-\theta\text{ in (\ref{vmat}),}
\] 
and the resulting element $g_a$ mapping $BC$ to $AD$ is given by
\begin{eqnarray}
g_a=\begin{pmatrix}\frac{b}{c}&-\frac{a^2}{bc}&\frac{a}{c}\sigma\\ -\frac{b}{c}&\frac{a^2}{bc}+\frac{c}{b}+\frac{a}{b}\sigma\theta&-\frac{a}{c}\sigma-\theta \\ -\frac{b}{c}\theta&-\frac{a}{b}\sigma+\frac{a^2}{bc}\theta &1+\frac{a}{c}\theta\sigma\\\end{pmatrix}.
 \end{eqnarray}

The group element $g_b$ mapping $A,B$ to $D,C$ is expressed by
$$
g_b=J\tilde{g}_a,
$$
where $J$ is the usual orthosymplectic matrix and $\tilde{g_a}$ is obtained from $g_a$ by interchanging the $a$ and $b$ in the matrix $g_a$. \medskip

As generators of the fundamental group, $g_a$, $g_b$ in fact depend upon the spin structure (cf.\/\cite{IPZ1})
as explicated in

\begin{theorem}\label{gab}
Consider the coordinates on $S\tilde{T}(F^1_1)$ illustrated in Figure~\ref{fig:coord11}. Suppose also that the spin structure is given by traversing the triangle ABC counterclockwise. Then generators of the fundamental group are given by
\begin{eqnarray}
g_a=\begin{pmatrix}\frac{b}{c}&-\frac{a^2}{bc}&\frac{a}{c}\sigma\\ -\frac{b}{c}&\frac{a^2}{bc}+\frac{c}{b}+\frac{a}{b}\sigma\theta&-\frac{a}{c}\sigma-\theta \\ -\frac{b}{c}\theta&-\frac{a}{b}\sigma+\frac{a^2}{bc}\theta &1+\frac{a}{c}\theta\sigma\\\end{pmatrix},\label{gb}\\
 g_b=J\begin{pmatrix}\frac{a}{c}&-\frac{b^2}{ac}&\frac{b}{c}\sigma\\ -\frac{a}{c}&\frac{b^2}{ac}+\frac{c}{a}+\frac{b}{a}\sigma\theta&-\frac{b}{c}\sigma-\theta \\ -\frac{a}{c}\theta&-\frac{b}{a}\sigma+\frac{b^2}{ac}\theta &1+\frac{b}{c}\theta\sigma\\\end{pmatrix}.\nonumber
 \end{eqnarray}
Provided that the spin structure has orientation on $c$ unchanged, for any reversal of spin orientation on $a$ or $b$, one must precompose the corresponding generator with $J^2\in OSp(1|2)$.
\end{theorem}

\subsection{Eigenvalues and geodesic lengths}

Super Fuchsian generators in $\OSp(1|2)$ were explicitly given in terms of coordinates on $S\tilde{T}(F_1^1)$ in Theorem \ref{gab}, and we begin by computing the eigenvalues of the
matrix $g_a$ expressed in (\ref{gb}), which are given by 
1 and  $r_a^{\pm 1}$ , where 
\begin{eqnarray}\label{rgen}
r_a+r_a^{-1}&=&\frac{a^2+b^2+c^2+abW_c+acW_b}{bc}=ah-W_a,\\
\text{and hence }r_a&=&\tfrac{1}{2}(ah-W_a+\sqrt{(ah-W_a)^2-4}).\nonumber 
\end{eqnarray}
The bosonic quantities $W_a,W_b,W_c$ are each equal to $\sigma\theta$ up to sign determined by spin structure (see \S\ref{sec:w}), and the super semi-perimeter $h$ is introduced in \S\ref{sec:h}.\medskip

Let us first explore the case when the spin structure  gives $W:=W_b=W_c=\sigma\theta$, which is the case for the matrices $g_a, g_b$ given in (\ref{gb}). 
This immediately leads to the  identities
\begin{eqnarray}
r_a^{\pm}(a^2+c^2-bcr_a^{\pm})=b(c-br_a^{\pm})-a(b+c)Wr_a^{\pm}, 
\end{eqnarray}
where we have set $r_a^{\pm}:= r_a^{\pm 1}$ for convenience. \medskip

One can then directly confirm the following

\begin{proposition}
\label{prop:eigenvectors}
The vectors 
\begin{eqnarray}
&&{\bf v}_{\pm}= 
\begin{pmatrix}
(a^2+c^2+acW-bcr_a^{\pm})(1-r_a^{\pm})-abr_a^{\pm}W\\ 
 b^2(1-r_a^{\pm})\\
ab\sigma+b(c-br_a^{\pm})\theta\\
 \end{pmatrix}\nonumber\\
 &&{\bf v}_0=\begin{pmatrix}
a(c-b)\sigma-a^2\theta\\ 
 ab\sigma+b(c-b)\theta\\
a^2+(b-c)^2\\
 \end{pmatrix}
 \end{eqnarray}
are the eigenvectors of  the matrix $g_a$ given in (\ref{gb}) with respective eigenvalues $r_a^{\pm}$ and $1$.  
\end{proposition}

The proposition immediately yields

\begin{theorem}
The eigenvalues of the generators  of the fundamental group $g_a$, $g_b$ are as follows:
\begin{itemize}
\item There is a common eigenvalue equal to $1$ for both $g_a$ and $g_b$,

\item If $str(g_{a}) >0$ (and $str(g_{b}) >0$, respectively) then the other two eigenvalues of $g_a$ (and $g_b$) are given by $ r_{a}^{\pm }$ (and $ r_{b}^{\pm }$),
\item If $str(g_{a}) <0$ (and $str(g_{b}) <0$, respectively) then the other two eigenvalues of $g_a$ (and $g_b$) are given by 
$ -r_{a}^{\pm }$ (and $ -r_{b}^{\pm }$).
\end{itemize}
\end{theorem}
\begin{proof}
Notice that the eigenvalues of $J^2g_a$ are the same as those of the matrix $g_a$ after a substitution $c,\theta\to -c,-\theta$.
This statement, together with formula (\ref{rgen}) proves the claim for the generator $g_a$.
For $g_b$, the claim follows upon applying  $J\in \OSp(1|2)$ to the quadrilateral $ABCD$ from (\ref{ABCD}), which interchanges points $A, C$.
\end{proof}

We are almost ready to formulate the main result of this section, relating the geodesic length to the group element which preserves it. 
We first note that geodesics, as described in this section, can always be mapped to the geodesic on the bosonic subspace of the light cone  (i.e.: where fermionic coordinates vanish). Therefore the corresponding group element is conjugated to the element in $\SL_2(\mathbb{R})$ subgroup, which then admits a ``normal form'' (through diagonalization) conducive to geodesic length computations via standard computations (see below). We emphasize that the ``length" here is not a positive real number, but is instead an even supernumber with positive body, since we work over some Grassmann algebra $\Lambda$.  

\begin{theorem}
The length $\ell_\gamma$ of the closed geodesic $\gamma$ preserved by the generator $g_a$  of the fundamental group is  related to its supertrace by the following formula
\footnote{We define the absolute value $|a|$ of a supernumber $a$ with a nonzero body $\epsilon(a)$ to be $a$ if $\epsilon(a)>0$ and $-a$ if $\epsilon(a)<0$.}
\begin{align}\label{el}
|{str} (g_a)+1|
=2\cosh(\tfrac{\ell_\gamma}{2})=
r_a+r_a^{-1}.
\end{align}
\end{theorem}
\begin{proof}
Proposition~\ref{prop:eigenvectors} asserts that the diagonalization of $g_a$ takes the form  
\[
\begin{pmatrix}\pm r_a&0&0\\ 0&\pm r^{-1}_a&0 \\ 0&0&1\\\end{pmatrix}\in \OSp(1|2),
\] 
and we note that this is an element of the subgroup $\SL_2(\mathbb{R})< \OSp(1|2)$ (we use the notation $\SL_2(\mathbb{R})$ to refer to the subgroup of special linear transformations with matrix entries from $\Lambda_0$). This diagonal matrix preserves the bosonic part (obtained by setting all fermionic coordinates to zero) of the super hyperboloid of two sheets in super-Minkowski space, as well as its projection to the pure bosonic part on super upper half-plane \footnote{We refer the reader to \cite{PZ} for a detailed discussion of the super hyperboloid, and the $\OSp(1|2)$-equivariant projection to the super upper-half plane.}. Moreover, this group element preserves the pure bosonic geodesic, generated by ${\bf e}=(1,0,0|0.0),~ {\bf f}=(0,1,0|0,0)$ (in the notation of Theorem \ref{Thm:geodesic}).\\

The result then follows
directly from the usual bosonic relationship between trace and eigenvalues (where we just replace $\mathbb{R}$ with $\Lambda_0$), namely, recall that in the classical bosonic case
with a hyperbolic $A\in \PSL_2({\mathbb R})$ corresponding to a closed geodesic of length $\ell$ in a Riemann surface, we have
\begin{align}
|{\rm trace} (A)|
=2\cosh(\tfrac{\ell}{2}).
\end{align}
Indeed, diagonalize $A$ as $\left(\begin{smallmatrix}\lambda &0\\ 0&\lambda^{-1} \end{smallmatrix}\right)$ and
consider the point $i$ on its axis with image $\lambda^2i$, so that
\begin{eqnarray}
\ell=\int^{\lambda^2}_1\frac{dx}{x}=2\ln\lambda,
\end{eqnarray}
whence $e^{\ell/2}+e^{-\ell/2}=\lambda+\lambda^{-1}={\rm trace}(A)$, as claimed. 
\end{proof}

\section{Super perimeter and $\lambda$-lengths}
\label{sec:superlamb}

For the remainder of this paper, unless otherwise explicitly stated, the surface under consideration is
the once-punctured torus $F=F_1^1$. To begin with, we fix an edge orientation representative for a given spin structure on $F$. This representative for the spin structure is allows us to define the $\lambda$-length coordinates for the corresponding connected component of $S\tilde{T}(F)$, and we hitherto refer to this representative simply as the spin structure. We set the $\lambda$-length coordinates as $a,b,c, \sigma, \theta$, so that the fundamental domain is parametrized as per Figure~\ref{fig:coord11}. Note here that we rotate the fundamental domain so that the spin orientation for $c$ agrees with that of Figure~\ref{fig:coord11}.

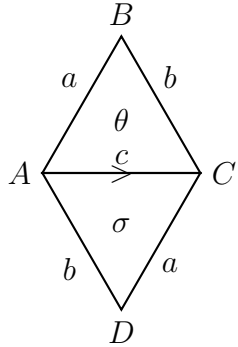
\begin{figure}[h!]

\centering

\begin{tikzpicture}[scale=0.7, baseline, thick]

\draw (0,0)--(3,0)--(60:3)--cycle;

\draw (0,0)--(3,0)--(-60:3)--cycle;

\draw node[above] at (70:1.5){$a$};

\draw node[above] at (30:2.8){$b$};

\draw node[below] at (-30:2.8){$a$};

\draw node[below=-0.1] at (-70:1.5){$b$};

\draw node[above] at (1.5,-0.05){$c$};

\draw node[above] at (1.5,-0.40){$>$};

\draw node[left] at (0,0) {};

\draw node[above] at (60:3) {};

\draw node[right] at (3,0) {};

\draw node[below] at (-60:3) {};

\draw node at (1.5,1){$\theta$};

\draw node at (1.5,-1){$\sigma$};

\draw node[left] at (0,0) {$A$};
\draw node[above] at (60:3) {$B$};
\draw node[right] at (3,0) {$C$};
\draw node[below] at (-60:3) {$D$};

\end{tikzpicture}
\caption{Coordinates for $\tilde{T}(F_1^1)$}
\label{fig:coord11}
\end{figure}

\subsection{$W$-invariant}
\label{sec:w}
For any ideal arc $e$ which is the shared diagonal of two ideal triangles in $F$, the spin structure on $F$ orients $e$ and allows us to designate the two adjoining ideal triangles as being either counter-clockwise or clockwise from $e$. Let $\theta_1$ and $\theta_2$ denote the $\mu$-invariants of the respective triangles clockwise and counter-clockwise from $e$ and define $W_e:=\theta_1\theta_2$. For example, $W_c=\sigma\theta$ in Figure~\ref{fig:coord11}.\medskip 

Although $W$-invariants are defined using an arbitrary initial ideal triangulation containing the edge $e$, the fact that the collection of all ideal triangulations containing $e$ are connected by $e$-preserving flips, coupled with the following lemma, show that $W$-invariants are well-defined.

\begin{lemma}
Given an ideal arc $e$, the quantity $W_e$, which is definitionally dependent on the spin structure orientation on $e$, is invariant under all flips which preserve $e$. 
\end{lemma}

\begin{proof}
Consider an arbitrary ideal triangulation $c,e,f$ used to define $W_e$. We use Figure~\ref{fig:winvariance} to show the invariance of $W_e$ under $e$-preserving flips. Specifically, we assume without loss of generality that we flip along $c$, oriented as per Figure~\ref{fig:winvariance}, and that $e$ and $f$ are either respectively equal to $a$ and $b$, or vice versa. This covers all possible topological configurations for $e$ with respect to the flipped edge.
\begin{figure}[h!]

\centering

\begin{tikzpicture}[scale=0.5, baseline,thick]

\draw (0,0)--(3,0)--(60:3)--cycle;

\draw (0,0)--(3,0)--(-60:3)--cycle;

\draw node[above] at (72:1.2){$\varepsilon_a$};

\draw node[above] at (23:2.9){$\varepsilon_b$};

\draw node[below] at (-22:2.9){$\varepsilon_a$};

\draw node[below=-0.1] at (-74:1.1){$\varepsilon_b$};

\draw node[above] at (1.5,-0.55){$>$};

\draw node[left] at (0,0) {};

\draw node[above] at (60:3) {};

\draw node[right] at (3,0) {};

\draw node[below] at (-60:3) {};

\draw node at (1.5,1.15){$\theta$};

\draw node at (1.5,-1){$\sigma$};

\draw node[above] at (1.5,-0.1){$c$};

\end{tikzpicture}
\begin{tikzpicture}[baseline]

\draw[->, thick](0,0)--(1,0);

\node[above]  at (0.5,0) {};

\end{tikzpicture}
\begin{tikzpicture}[scale=0.5, baseline, thick]

\draw (0,0)--(60:3)--(-60:3)--cycle;

\draw (3,0)--(60:3)--(-60:3)--cycle;

\draw node[above] at (72:1.2){$\varepsilon_a$};

\draw node[above] at (23:2.9){$-\varepsilon_b$};

\draw node[below] at (-22:2.9){$\varepsilon_a$};

\draw node[below=-0.1] at (-74:1.1){$-\varepsilon_b$};

\draw node[left] at (2.0941,0){${\mathbf{\wedge}}$};

\draw node[left] at (0,0) {};

\draw node[above] at (60:3) {};

\draw node[right] at (3,0) {};

\draw node[below] at (-60:3) {};

\draw node at (0.8,0){$\theta'$};

\draw node at (2.2,0){$\sigma'$};

\draw node[left] at (1.8,1){$c'$};
\end{tikzpicture}
\caption{Diagram for checking the invariance of $W$-invariants.}
\label{fig:winvariance}
\end{figure}
When $e=a$, flipping in $c$ preserves the orientation of $e=a$. In addition, $\theta$ is clockwise from $a$ with orientation $\varepsilon_a$ iff. $\theta'$ is clockwise from $a$ with orientation $\varepsilon_a$ (identical statement for $\sigma$ and $\sigma'$). Therefore, $W_{e=a}$ is preserved. On the other hand, when $e=b$, flipping in $c$ changes the orientation of $e=b$. This means that $\theta$ is clockwise from $b$ with orientation $\varepsilon_b$ iff. $\theta'$ is clockwise from $b$ with orientation $-\varepsilon_b$ (identical statement for $\sigma$ and $\sigma'$). Thus, $W_{e=b}$ is also preserved under flipping in $c$. This covers all possible toplogical configurations.

\end{proof}

\subsection{Super (semi-)perimeter}
\label{sec:h}

Recall that the (classical) perimeter, namely, the sum $2({a\over{bc}}+{b\over{ac}}+{c\over{ab}})$
of the $h$-lengths, is invariant under flips. We now generalize this construction for super-tori
and
define
the \emph{super semi-perimeter}
\[
h=\frac{a}{bc}+\frac{b}{ac}+\frac{c}{ab}
+\left(\frac{W_a}{a}+\frac{W_b}{b}+\frac{W_c}{c}\right),
\]
Note in particular that the super semi-perimeter $h$ is symmetric in $a,b,c$.

\begin{proposition} \label{prop:ss} The super semi-perimeter is invariant under flips.
\end{proposition}

\begin{proof}
Suppose that the flip on ideal arc $c$ produces $d$, 
\begin{figure}[h!]

\centering

\begin{tikzpicture}[scale=0.7, baseline, thick]

\draw (0,0)--(3,0)--(60:3)--cycle;

\draw (0,0)--(3,0)--(-60:3)--cycle;

\draw node[above] at (70:1.5){$a$};

\draw node[above] at (30:2.8){$b$};

\draw node[below] at (-30:2.8){$a$};

\draw node[below=-0.1] at (-70:1.5){$b$};

\draw node[above] at (1.5,-0.05){$c$};

\draw node[above] at (1.5,-0.40){$>$};

\draw node[left] at (0,0) {};

\draw node[above] at (60:3) {};

\draw node[right] at (3,0) {};

\draw node[below] at (-60:3) {};

\draw node at (1.5,1){$\theta$};

\draw node at (1.5,-1){$\sigma$};

\end{tikzpicture}
\begin{tikzpicture}[baseline]

\draw[->, thick](0,0)--(1,0);

\node[above]  at (0.5,0) {};

\end{tikzpicture}
\begin{tikzpicture}[scale=0.7, baseline, thick]

\draw (0,0)--(60:3)--(-60:3)--cycle;

\draw (3,0)--(60:3)--(-60:3)--cycle;

\draw node[above] at (70:1.5){$a$};

\draw node[above] at (30:2.8){$b$};

\draw node[below] at (-30:2.8){$a$};

\draw node[below=-0.1] at (-70:1.5){$b$};

\draw node[left] at (1.6,1){$d$};

\draw node[left] at (1.92,0){${\mathbf{\wedge}}$};

\draw node[left] at (0,0) {};

\draw node[above] at (60:3) {};

\draw node[right] at (3,0) {};

\draw node[below] at (-60:3) {};

\draw node at (0.8,0){$\theta'$};

\draw node at (2.2,0){$\sigma'$};

\end{tikzpicture}
\end{figure}
where the orientation of arrows is as in the Proposition~\ref{cor:storus}.
The super Ptolemy equation therefore gives 
$dc=a^2+b^2+abW_c$, and we divide by $abc$ to produce
\begin{align}
\frac{d}{ab}=\frac{a}{cb}+\frac{b}{ac}+\frac{W_c}{c}.\label{eq:sptol}
\end{align}
We must confirm that
\[
\frac{a}{bc}+\frac{b}{ac}+\frac{c}{ab}
+\left(\frac{W_a}{a}+\frac{W_b}{b}+\frac{W_c}{c}\right)
=\frac{d}{ab}+\frac{a}{bd}+\frac{b}{ad}
+\left(\frac{W_d}{d}+\frac{W_b}{b}+\frac{W_a}{a}\right),
\] 
and according to equation~\eqref{eq:sptol}, the right-hand side is equal to
\[
\left(\frac{b}{a}+\frac{a}{b}\right)\frac{1}{d}
+\frac{a}{cb}+\frac{b}{ac}+\frac{W_c}{c}
+\left(\frac{W_b}{b}+\frac{W_a}{a}+\frac{W_d}{d}\right).
\]
Thus, the difference of the lefthand and righthand sides is given by
\[
\frac{c}{ab}-\left(\frac{a^2+b^2}{ab}\right)\frac{1}{d}-\frac{W_d}{d},
\]
which vanishes by the super Ptolemy equation.
\end{proof}

\subsection{Limiting behavior of $\lambda$-lengths}
\label{sec:split}
The remainder of this section is dedicated to a key computation regarding the behavior of $\lambda$-lengths under repeated Dehn twists. Endow $F$ with an auxiliary hyperbolic metric and consider an arbitrary ideal geodesic on $F$ labeled by $a$ and a homotopy class $\gamma_a\in\pi_1(F)$ whose geodesic representative is disjoint from $a$. We formulate the aforementioned computation in terms of the infinite-cyclic cover $S$ of the surface obtained from cutting $F$ along $a$, where the covering group is generated by $\gamma_a\in\pi_1(F)$.  The total space of this cover is homeomorphic to the bi-infinite strip 
\[
S=\left\{
(x,y)\in \mathbb{R}^2-\mathbb{Z}^2 \mid 0\leq x\leq 1
\right\}
\]
with the deleted $\mathbb{Z}^2$ corresponding to lifts of the puncture.  For all $m,n\in\mathbb{Z}$, we label the diagonal segments connecting
$(0,m)$ to $(1,n+m)$ by
$b_{-n}$. 
In particular, we may respectively set the $b,c,d$ terms used earlier in this section to be $b_n=b$, $b_{n-1}=c$ and $b_{n+1}=d$, for some $n$.

\begin{figure}[h!]
\includegraphics[scale=1.5]{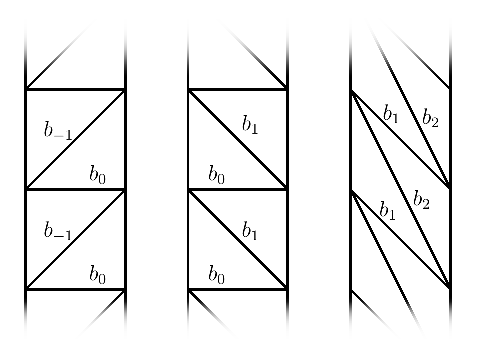}
\caption{A sequence of $b_n$, $n\in\mathbb{Z}$, related by flips and/or Dehn twists. The vertical edges all correspond to lifts of the ideal geodesic labeled by $a$.}
\label{fig:flipsequence}
\end{figure}

Given the above notation (see Figure~\ref{fig:flipsequence}), a flip in $b_{n-1}$ takes $a,b_{n-1},b_n$ respectively to $a,b_{n+1},b_n$. Following this with a flip in $b_n$, takes $a,b_{n+1},b_n$ to $a,b_{n+1},b_{n+2}$. The composition of these two flips corresponds to the action of a double (right) Dehn twist along $\gamma_a$ lifted to $S$.\medskip

We solve the following recursion
\begin{align}
h=\tfrac{b_{n+1}}{ab_n}+\tfrac{b_{n-1}}{ab_n}
+\left(\tfrac{W_a}{a}+\tfrac{W_{b_n}}{b_n}\right),\label{eq:sperimeter}
\end{align}
which arises from the definition of $h$ and the following identity
\[
\frac{b_{n+1}}{ab_n}+\frac{b_{n-1}}{ab_n}+\frac{W_{b_n}}{b_n}
=\frac{b_{n+1}}{ab_n}+\frac{a}{b_nb_{n+1}}+\frac{b_{n}}{ab_{n+1}}
+\left(\frac{W_{b_n}}{b_n}+\frac{W_{b_{n+1}}}{b_{n+1}}\right).
\]
The latter identity in turn comes from equation~\eqref{eq:sptol} by taking $b=b_n$, $c=b_{n+1}$ and $d=b_{n-1}$ to get
\[
\frac{b_{n-1}}{ab_n}
=\frac{a}{b_nb_{n+1}}+\frac{b_{n}}{ab_{n+1}}
+\frac{W_{b_{n+1}}}{b_{n+1}}.
\]
Thus, equation~\eqref{eq:sperimeter} gives the 3-term recursion
\begin{align}
0=b_{n+1}-ab_n(h-\tfrac{W_a}{a})+b_{n-1}+aW_{b_n}.\label{eq:recursion}
\end{align}

Although \eqref{eq:recursion} seems at first glance to display sophisticated dependence on $W_{b_n}$, Equation~\eqref{eq:muinvars} tells us that 
\[
W_{b_{n-1}}=W_{b_{n+1}},
\]
and hence the sequence $\{W_{b_n}\}$ is $2$-cyclic. Moreover, the fact that there are only two triangles for any ideal triangulation of a once-punctured torus means that $W_{b_n}$ and $W_{b_{n+1}}$ are equivalent up to sign, and hence the $\{W_{b_n}\}$ either remain constant (i.e., $W_{b_n}=W_{b_{n-1}}$) or oscillate in sign (i.e., $W_{b_n}=-W_{b_{n-1}}$). We give both of these cases (by depicting the relevant $n=0,1$ cases) in Figure~\ref{fig:oscillate}. Specifically, the left figure illustrates depicts the general configuration for when the $\{W_{b_n}\}$ remain constant (i.e.: $W_{b_0}=W_{b_1}$). Note that the orientation on $a$ here is arbitrary, and we have (up to rotation) placed $b_0$ in this specific configuration. The right figure likewise depicts the general configuration for when the sequence oscillates in sign (i.e.: $W_{b_0}=-W_{b_1}$).
\begin{figure}[h!]

\centering

\begin{tikzpicture}[scale=0.7, baseline, thick]

\draw (0,0)--(3,0)--(60:3)--cycle;

\draw (0,0)--(3,0)--(-60:3)--cycle;

\draw node[above] at (70:1.5){$a$};

\draw node[above] at (30:2.8){$b_1$};

\draw node[below] at (-30:2.8){$a$};

\draw node[below=-0.1] at (-70:1.5){$b_1$};

\draw node[above] at (1.5,-0.05){$b_0$};

\draw node[above] at (1.5,-0.40){$>$};

\draw (2.43,1) node[rotate=-60] {$>$};

\draw (0.58,-1) node[rotate=-60] {$>$};

\draw node[left] at (0,0) {};

\draw node[above] at (60:3) {};

\draw node[right] at (3,0) {};

\draw node[below] at (-60:3) {};

\draw node at (1.5,1){$\theta$};

\draw node at (1.5,-1){$\sigma$};

\end{tikzpicture}
\begin{tikzpicture}[baseline]


\node[above]  at (0.5,0) {};

\end{tikzpicture}
\begin{tikzpicture}[scale=0.7, baseline, thick]

\draw (0,0)--(3,0)--(60:3)--cycle;

\draw (0,0)--(3,0)--(-60:3)--cycle;

\draw node[above] at (70:1.5){$a$};

\draw node[above] at (30:2.8){$b_1$};

\draw node[below] at (-30:2.8){$a$};

\draw node[below=-0.1] at (-70:1.5){$b_1$};

\draw node[above] at (1.5,-0.05){$b_0$};

\draw node[above] at (1.5,-0.40){$>$};

\draw (2.43,1) node[rotate=120] {$>$};

\draw (0.58,-1) node[rotate=120] {$>$};

\draw node[left] at (0,0) {};

\draw node[above] at (60:3) {};

\draw node[right] at (3,0) {};

\draw node[below] at (-60:3) {};

\draw node at (1.5,1){$\theta$};

\draw node at (1.5,-1){$\sigma$};

\end{tikzpicture}

\caption{Constant (left) vs oscillatory (right) $\{W_{b_n}\}$.}
\label{fig:oscillate}
\end{figure}
In any case, this is a very mild dependence on the spin structure, and for all analytical purposes means that $W_{b_n}$ is dwarfed by the behavior of other terms and hence essentially behaves as a constant. In particular, we may apply standard difference equation techniques.\medskip

The solution for the homogeneous part of equation~\eqref{eq:recursion} takes the form $b_n^{\mathrm{hom}}=xr^n+yr^{-n}$, where the conditions on the bodies (see Section~\ref{sec:grassmann}) $\epsilon(x), \epsilon(y)$ of $x, y$ correspondingly, are as follows: $\epsilon(x),\epsilon(y)>0$, and
\begin{align}
r&=\tfrac{1}{2}(ah-W_a+\sqrt{(ah-W_a)^2-4}).\label{eq:rvalue}\\
&=\left(\frac{ah+\sqrt{a^2h^2-4}}{2}\right)
\left(1-\frac{W_a}
{\sqrt{a^2h^2-4}}
\right).
\end{align}
Moreover, a particular solution is given by $b_n^{\mathrm{part}}
=\tfrac{aW_{b_{n}}}{ah-W_{a}\mp2}=\tfrac{aW_{b_{n}}}{ah\mp 2}$, where the minus sign is taken if $W_{b_n}=W_{b_{n-1}}$ and the plus sign is taken if $W_{b_n}=-W_{b_{n-1}}$. Therefore, the complete solution is given by
\begin{align}
b_n
=b_n^{\mathrm{hom}}+b_n^{\mathrm{part}}
=xr^n+yr^{-n}+\tfrac{aW_{b_{n}}}{ah\mp 2},\label{eq:recursol}
\end{align}
for $r$ given in equation~\eqref{eq:rvalue} where $\epsilon(x),\epsilon(y)>0$.

\section{The spectrum of super $\lambda$-lengths}
\label{sec:bodysoul}

\S\ref{sec:dualcomplex} to \S\ref{sec:asymptotics} are straight-forward adaptations of Bowditch's work \cite{Bpro, Bmar} to the super Riemann surface context. We introduce novel analytical ideas in \S\ref{sec:comparison} to establish the body-soul comparison (Theorem~\ref{thm:comparison}), which is central to the proof of the super McShane identity.

\subsection{Dual complex}\label{sec:dualcomplex}

Theorem~\ref{thm:icd} asserts the existence of an ideal triangulation of $T(F)=\mathbb{D}$  given by the Farey tessellation $\tau_*$ of $\mathbb{D}$ (see $\S$~\ref{sec:fareytess}). The theory further states that each face of this ideal triangulation is labeled by a distinct ideal triangulation of $F$, each edge in $\tau_*$ is labeled by a distinct non-maximal  ideal cell-decomposition of $F$, and each ideal vertex correponds to a distinct ideal arc on $F$.\medskip

On the other hand, the set of ideal arcs on $F$ is topologically ``dual'' to the set of (isotopy classes of) simple closed curves on $F$ in the sense that for each ideal arc $a$ there is a unique simple closed curve $\gamma_a$ which is disjoint from $a$, and vice versa. This duality suggests the following abstract \emph{dual ``complex"} $\Omega$ to the Farey tessellation $\tau_*$
\begin{itemize}
\item
the faces $\mathcal{F}(\Omega)$ of $\Omega$ are labeled by simple closed curves on $F$;
\item
the edges $\mathcal{E}(\Omega)$ of $\Omega$ are labeled by pairs of simple closed curves on $F$ which transversely intersect once;
\item
the vertices $\mathcal{V}(\Omega)$ of $\Omega$ are labeled by triples of simple closed curves on $F$ which pairwise transversely intersect once.
\end{itemize}

We clarify that $\Omega$ is not a cell-complex as each face in $\mathcal{F}(\Omega)$ is bordered by an infinite sequence of distinct vertices and edges. However, we {can} nonetheless na\"ively geometrically realize $\Omega$ in the following way

\begin{itemize}
\item 
let the vertices $\mathcal{V}(\Omega)$ be the centers of the ideal triangles constituting $\tau_*$;
\item 
let the edges $\mathcal{E}(\Omega)$ be the geodesic arcs joining the centers of adjacent ideal triangles in $\tau_*$;
\item
let the faces $\mathcal{F}(\Omega)$ be the connected components of the complement $\mathbb{D}-(\mathcal{V}(\Omega)\cup \mathcal{E}(\Omega))$.
\end{itemize}
This embeds $\Omega$ in $\mathbb{D}$ in a manner dual to $\tau_*\subset\mathbb{D}$ (see Figure~\ref{fig:farey}). We regard $\mathcal{F}(\Omega)$ as the set of \emph{complementary regions}  \cite{Bpro} to $\tau_*$ in $\mathbb{D}$.

\begin{figure}[h!]
\includegraphics[scale=0.5]{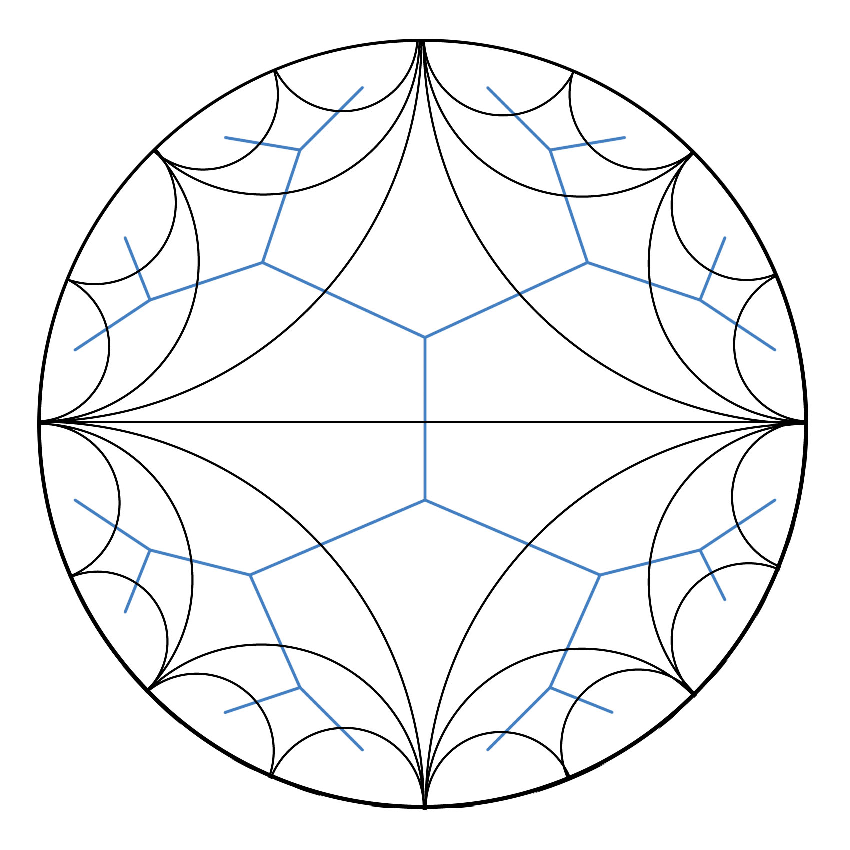}
\caption{The Farey tessellation $\tau_*$ (black) and its dual complex $\Omega$ (blue).}
\label{fig:farey}
\end{figure}

\subsection{Super Markoff map}
Every element of the decorated super Teichm\"uller space $S\tilde{T}(F)$ defines a function $\phi:\mathcal{F}(\Omega)\to \mathbb{R}_{S[N]}$ by assigning to a complementary region $\mathbf{x}\in\mathcal{F}(\Omega)$ the super $\lambda$-length $x=:\phi(\mathbf{x})$ of the ideal arc dual to $\mathbf{x}$. We refer to any such function $\phi$ as a \emph{super Markoff map} and note that it necessarily satisfies
\begin{itemize}
\item the \emph{vertex relation} (figure~\ref{fig:proof}): for any three complementary regions $\mathbf{a},\mathbf{b},\mathbf{c}$ meeting at the same vertex
\begin{align}
a^2+b^2+c^2+abW_c+acW_b+bcW_a=habc,\label{eq:vertex}
\end{align}
as follows from the definition of super semi-perimeter $h$, and
\item the \emph{edge relation} (figure~\ref{fig:proof}): for complementary regions $\mathbf{a},\mathbf{b},\mathbf{c},\mathbf{d}$ arranged so that $\mathbf{a},\mathbf{d}$ touch the two ends of an edge lying on the boundary of both $\mathbf{b}$ and $\mathbf{c}$,
\begin{align}
a+d+(bW_c+cW_b)=hbc\label{eq:edge}
\end{align}
as follows from (\ref{eq:vertex}) and the super Ptolemy relation 
\[
ad=b^2+c^2+bcW_a.
\] 
\end{itemize}

\subsection{Finite subtree sum}

Given a subtree $T$ of the $1$-skeleton of the dual complex, we define
\begin{itemize}
\item $\mathcal{C}(T)$ to be the set of oriented edges $\vec{e}\in\vec{\mathcal{E}}(\Omega)$ with its head in $T$ but its tail outside of $T$;
\item $\mathcal{F}(T)$ to be the set of complementary regions whose closure intersects with $T$. 
\end{itemize}

Given an oriented edge $\vec{e}\in\vec{\mathcal{E}}(\Omega)$ positioned with
\begin{itemize}
\item head directed towards $\mathbf{c}$,
\item $\mathbf{a}$ and $\mathbf{b}$ on the two sides of $\vec{e}$, 
\end{itemize}
we define the function 
\begin{align}
\psi(\vec{e}):=\tfrac{1}{h}(\tfrac{c}{ab}+\tfrac{W_a}{2a}+\tfrac{W_b}{2b}).
\end{align}
Using this notation, the vertex relation \eqref{eq:vertex} and the edge relation \eqref{eq:edge} respectively become
\begin{align}
\psi(\vec{e}_1)+\psi(\vec{e}_2)+\psi(\vec{e}_3)=1
\quad \text{and}\quad 
\psi(\vec{e})+\psi(-\vec{e})=1,
\end{align}
where $\vec{e}_1,\vec{e}_2,\vec{e}_3$ denote oriented edges which point at the same vertex, and where $-\vec{e}$ denotes $\vec{e}$ endowed with the opposite orientation. 
Given the relations on the function $\psi$ above, we arrive at the following property, which is easily confirmed inductively.
\begin{proposition}
\label{thm:finitesum}
Given any finite subtree $T$, we have
\begin{align}
\sum_{\vec{e}\in C(T)}\psi(\vec{e})=1.\label{eq:finsum}
\end{align}
\end{proposition}

Our goal is to show that as the finite tree $T$ extends to the full infinite trivalent tree $T_\infty$, equation ~\eqref{eq:finsum} tends to the desired super McShane identity. The philosophical ``$C(T_\infty)$'' consists of $3\times 2^{\mathbb{N}}$ infinite ends of $T_\infty$, with countably many ``rational ends'' whose summands give rise to the summands in the super McShane identity, and uncountably many ``irrational ends'' whose summands deteriorate so quickly as to provide zero contribution in the limit \cite{labtan}. Much of the remainder of this paper is dedicated to translating this intuition into concrete analysis.

\subsection{Oriented trivalent tree}
\label{sec:orientedtree}

In \cite{Bpro}, orientations are assigned to the edges $\mathcal{E}(\Omega)$ of the dual complex $\Omega$ so that an edge points from $\mathbf{a}$ to $\mathbf{d}$ if any of the following equivalent conditions are satisfied
\begin{align}
[\epsilon(a)<\epsilon(d)]
\Leftrightarrow
[\tfrac{\epsilon(a)}{\epsilon(b)\epsilon(c)}<\tfrac{\epsilon(d)}{\epsilon(b)\epsilon(c)}]
\Leftrightarrow
[\tfrac{\epsilon(a)}{\epsilon(b)\epsilon(c)}<\tfrac{\epsilon(h)}{2}]
\Leftrightarrow
[\tfrac{\epsilon(d)}{\epsilon(b)\epsilon(c)}>\tfrac{\epsilon(h)}{2}],
\end{align}
where the last three conditions are equivalent thanks to the body of the edge relation~\eqref{eq:edge}. Edges which do not acquire an orientation satsify the equivalent conditions
\begin{align}
[\epsilon(a)=\epsilon(d)]
\Leftrightarrow
[\tfrac{\epsilon(a)}{\epsilon(b)\epsilon(c)}=
\tfrac{\epsilon(d)}{\epsilon(b)\epsilon(c)}=\tfrac{\epsilon(h)}{2}]\label{eq:nongeneric}
\end{align}
and are arbitrarily assigned an orientation. We call these \emph{flexible edges}.\medskip

The vertex condition tells us that
\begin{align}
\epsilon(\tfrac{a}{bc})+\epsilon(\tfrac{b}{ac})+\epsilon(\tfrac{c}{ab})=\epsilon(h),\label{eq:littlesum}
\end{align}
and this contradicts the possibility of having vertices with $2$ or more outgoing arrows as such a configuration would lead to the left hand side of \eqref{eq:littlesum} being strictly greater than $\tfrac{\epsilon(h)}{2}+\tfrac{\epsilon(h)}{2}=\epsilon(h)$. The restriction of only having \emph{sink} vertices (i.e., $3$ incoming edges) and \emph{fork} vertices (i.e., $2$ incoming edges) means that the resultant oriented trivalent tree either
\begin{enumerate}
\item
has one sink vertex, with every edge pointing towards it;
\item
only has fork vertices, with every edge pointing towards an ``ideal boundary'' point of the infinite trivalent tree.
\end{enumerate}
This latter scenario implies that there is a strictly decreasing sequence of (classical) $\lambda$-lengths. However, the non-discreteness of this sequence then contradicts
the discreteness of the length/trace spectrum of the underlying Fuchsian representation for our given super Fuchsian representation $\rho$ thanks to the fact\footnote{This is an elementary exercise in hyperbolic trigonometry, e.g.: \cite[Note 3.4]{huang_thesis}.} that
\begin{align}
\epsilon(ah)
=\epsilon(\mathrm{tr}(\rho(\gamma_a))).\label{eq:lambdavstrace}
\end{align}
In short, only scenario $(1)$ arises.

\begin{remark}\label{rmk:nongeneric}
Super Markoff maps which admit ``flexible" edges satisfying the (real) codimension $1$ algebraic conditions \eqref{eq:nongeneric}  are hence non-generic. Furthermore,
the impossibility of allowing vertices with $2$ outgoing arrows implies that
\begin{itemize}
\item
we cannot have two adjacent flexible edges;
\item
all edges immediately adjacent to a flexible edge $e$ must point toward it, and hence
\item
there is at most one flexible edge, as the shortest path between two disconnected flexible edges necessarily contains a vertex with $2$ outgoing arrows, and hence
\item
super Markoff maps admit at most one flexible edge, and the sink vertex necessarily bounds it.
\end{itemize}
\end{remark}

In fact, super Markoff maps which have flexible edges are precisely those which arise from decorated super Fuchsian representations whose underlying (body) Fuchsian representation lies on the edges of the Farey triangulation $\tau_*$ of $T(F)$.

\subsection{Bounded trace regions}
\label{sec:boundedtraceregions}

Given $\phi$ derived from a decorated marked once-punctured super torus, for every $m\geq0$, define the following collections of complementary regions
\begin{align}
\Omega_\phi(m):=
\left\{
\mathbf{a}\in \mathcal{F}(\Omega)\mid 
\epsilon(ah)=\epsilon(a)\epsilon(h)\leq m
\right\}.
\end{align}

To better understand the shapes of these sets, we utilize the following weak corollary of \cite[Theorem~1]{Bmar}:

\begin{proposition}\label{thm:bow}
If $\phi$ is a super Markoff map derived from a decorated marked once-punctured super torus, then
\begin{enumerate}
\item
 $\Omega_\phi(3)$ is non-empty, and
\item
(the closure of) the collection of complementary regions in $\Omega_\phi(m)$ is the connected union of finitely many (closures of) complementary regions. 
\end{enumerate}
\end{proposition}

We provide a proof here for completeness and emphasize that the following argument is essentially repeated from \cite{Bmar} and only applies to a specialized subcase of the treatment there.

\begin{proof}
Recall that \eqref{eq:lambdavstrace} relates the classical $\lambda$-lengths $\epsilon(a)$ of complementary regions $\mathbf{a}$ to the lengths of simple closed geodesics $\gamma_a$ of the underlying Fuchsian representation for a given super Markoff map. This ensures the discreteness of the $\{\epsilon(a)\}$ spectrum and hence the finiteness of the number of complementary regions in each $\Omega_\phi(m)$. Connectedness follows from the fact that the edges of $\Omega$ adjacent to (but not contained in) any component $\Omega_1,\ldots,\Omega_k$ of $\Omega_\phi(m)$ must point into $\Omega_i$. This then forces the shortest path between distinct components $\Omega_i$ and $\Omega_j$ to contain a vertex with $2$ outgoing edges, and this is impossible as it contradicts the vertex condition.\medskip

Finally, the fact that $\Omega_\phi(3)$ is non-empty is equivalent to the fact that the maximum of the systole function on the moduli space of once-punctured hyperbolic tori is equal to $2\,\mathrm{arcosh}(\frac{3}{2})$. Consider the complementary regions $\mathbf{a},\mathbf{b},\mathbf{c}$ surrounding the sink vertex and take without loss of generality
\[
0<x:=\tfrac{\epsilon(ah)}{\epsilon(bh)\epsilon(ch)}
\leq
y:=\tfrac{\epsilon(bh)}{\epsilon(ah)\epsilon(ch)}
\leq
z:=\tfrac{\epsilon(ch)}{\epsilon(ah)\epsilon(bh)}
\leq
\tfrac{1}{2}.
\]
If $yz<\frac{1}{9}$, then we obtain the following contradiction:
\[
1=x+y+z\leq 2y+z <\tfrac{2}{9z}+z\leq \tfrac{4}{9}+\tfrac{1}{2}<1.
\]
Therefore, $3\geq\frac{1}{\sqrt{yz}}=\epsilon(ah)$, and $\Omega_\phi(3)\neq\varnothing$ as claimed.
\end{proof}

\subsection{The asymptotics of neighboring domains}
\label{sec:asymptotics}

Given an arbitrary complementary region $\mathbf{a}\in F(\Omega)$ consider the following two sequences of complementary regions
\begin{itemize}
\item
the sequence $\{\mathbf{b}_i\}_{i\in\mathbb{Z}}$ of complementary regions adjacent to $\mathbf{a}$, sequentially indexed around $\mathbf{a}$ (see Figure~\ref{fig:spiral}), and 
\item
the sequence $\{\mathbf{c}_i\}_{i\in\mathbb{Z}}$ of complementary regions which are precisely one edge away from $\mathbf{a}$, indexed so that $\mathbf{c}_i$ is wedged in between $\mathbf{b}_i$ and $\mathbf{b}_{i+1}$ (see Figure~\ref{fig:spiral}).
\end{itemize}

\begin{figure}[h!]
\includegraphics[scale=0.8]{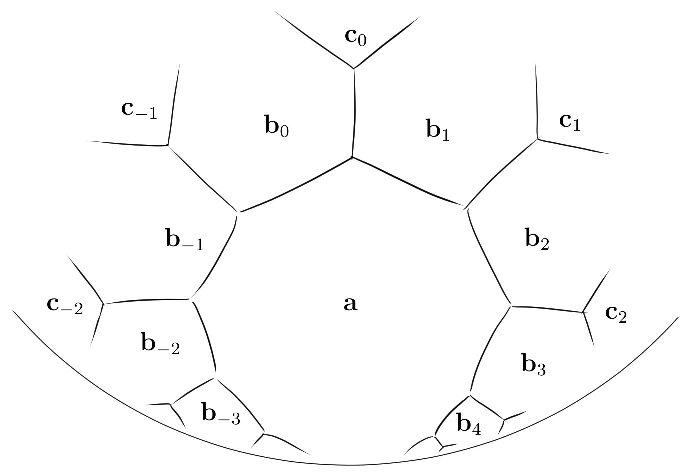}
\caption{Complementary regions surrounding a complementary region $\mathbf{a}$.}
\label{fig:spiral}
\end{figure}

We have in fact already seen the sequence $\{\mathbf{b}_i\}$ in $\S$~\ref{sec:split} as the sequence of complementary regions corresponding to diagonals with super $\lambda$-lengths $\{b_i\}$. In keeping with this notation, we define the sequence $\{c_i=\phi(\mathbf{c}_i)\}$ of super $\lambda$-lengths associated to the arcs $\mathbf{c}_i$.

\begin{lemma}\label{thm:spiralbound}
The asymptotic growth rate of $\{b_i\}$ and $\{c_i\}$ respectively satisfy
\[
\|s_{2k}(b_i)\|=O(|i|^kR^{|i|})
\text{ and }
\|s_{2k}(c_i)\|=O(|i|^{2k}R^{2|i|})
\] as $|i|$ approaches infinity, where $R=\epsilon(r)>1$ is the body of the $r$ given in equation~\eqref{eq:rvalue}.
\end{lemma}

\begin{proof}
This is verified by direct computation from equation~\eqref{eq:recursol}. We first note that the right-most term in \eqref{eq:recursol} oscillates among finitely many different values and effectively behaves as a constant term. In particular, its contribution becomes (comparatively) negligible as $|i|\to\infty$ and is thus ignored. To understand the behavior of $\|s_{2k}(b_i)\|$, we first express $r$ as $s_0(r)+s_2(r)+\ldots+s_{2\lfloor\frac{N}{2}\rfloor}(r)$, where $N$ is the number of variables used to generate the Grassmann algebra $\mathbb{R}_{S[N]}$. Then, $s_{2k}(r^i)$ is a polynomial in $\{s_{2j}(r)\}_{j=0,1,\ldots,k}$. As an example:
\begin{align*}
s_6(r^i)&=
{i\choose 1}s_0(r)^{i-1}s_6(r)
+
{i\choose 2}s_0(r)^{i-2}s_2(r)s_4(r)
+
{i\choose 3}s_0(r)^{i-3}s_2(r)^3\notag\\
&=
iR^{i-1}s_6(r)
+
\tfrac{i(i-1)}{2}R^{i-2}s_2(r)s_4(r)
+
\tfrac{i(i-1)(i-2)}{6}R^{i-3}s_2(r)^3.
\end{align*}
We make two observations:
\begin{itemize}
\item
the number of monomials is bounded above by the number of ordered partitions of $k$ (and hence independent of $i$); 
\item
the $N+1$ nilpotence of soul terms in $\mathbb{R}_{S[N]}$ means that most of the monomials will be $0$. That is to say, only the factor of $s_0(r)=R$ in each monomial is permitted to arbitrarily increase in its exponent as $i\to\infty$.
\end{itemize}
As an example, we see from the previous example computation that $\|s_6(r^i)\|$ is dominated by the $\tfrac{i(i-1)(i-2)}{6}R^{i-3}s_2(r)^3$ term as $i\to\infty$, and hence $\|s_6(r^i)\|=O(i^3R^i)$ as $i\to\infty$. More generally speaking, $\|s_{2k}(r^i)\|$ is dominated by the ${i\choose k}R^{i-k} s_2(r)^k$ monomial as $i\to\infty$, and hence  $\|s_{2k}(r^i)\|=O(i^kR^i)$ as $i\to\infty$. In contrast, as $i\to-\infty$, the two bulleted observations above tell us that $\|s_{2k}(r^i)\|$ tends to $0$. Instead, we see that 
\[
\|s_{2k}(r^{-i})\|=O(|i|^k(R^{-1})^{-i})=O(|i|^kR^{i}),\text{ as }i\to-\infty.
\]
Pooling these conclusions with \eqref{eq:recursol}, we see that $\|s_{2k}(b_i)\|=O(|i|^kR^{|i|})$. The asymptotic behavior for $\|s_{2k}(c_i)\|$ follows, as
\[
c_i=a^{-1}(b_i^2+b_{i+1}^2+b_ib_{i+1}W_{c_i})=O(|i+1|^{2k}R^{2|i+1|})=O(|i|^{2k}R^{2|i|}).
\]

\end{proof}

\subsection{Body--soul comparison}
\label{sec:comparison}

The aim of this subsection is to establish the following 

\begin{theorem}[Body--soul comparison]
\label{thm:comparison}
For every $\delta>0$, there exists a constant $M=M_{\delta,\phi}\geq1$ such that
\[
\|s(a)\|\leq M\cdot|\epsilon(a)|^{1+\delta}
\]
for every super $\lambda$-length $a$.
\end{theorem}

\begin{remark}
The body--soul comparison as stated here is for finite-dimensional Grassmann algebras $\mathbb{R}_{S[N]}$ as $M$ is dependent on $N$. We note also that the infinite-dimensional context is more subtle, and there are various approaches in the literature (see, for example, \cite[Definition~3.1.4 and \S3.2]{Rbook}).
\end{remark}

\begin{definition}
Consider an arbitrary complementary region $\mathbf{a}\in \mathcal{F}(\Omega)$, and denote its corresponding super $\lambda$-length by $a=\phi(\mathbf{a})$. Let $\mathbf{a}^*$ denote any complementary region adjacent or equal to $\mathbf{a}$ with the second smallest ({\sl a priori} counted with multiplicity) classical $\lambda$-length with respect to $\epsilon(\phi)$ and define $\epsilon(a^*):=\epsilon(\phi(\mathbf{a}^*))$.
\end{definition}

\begin{lemma}
\label{thm:technicalcompare}
For every $\delta>0$, there are positive real numbers \\
$M_{\lfloor{N/2}\rfloor}\geq\ldots\geq M_1\geq M_0=1$ such that 
\begin{align}
\|s_{2k}(a)\|\leq M_{k} \cdot \epsilon(a) \cdot \epsilon(a^*)^{k\delta},\label{eq:bound}
\end{align}
for the $\lambda$-length $a=\phi(\mathbf{a})$ of any $\mathbf{a}\in \mathcal{F}(\Omega)$.
\end{lemma}

\begin{proof}
Fix an arbitrary $\delta>0$, and choose any positive real number $Y$ large enough so that $Y^{\frac{1}{\delta}}\geq \tfrac{3}{2}$. We consider the following collection of subsets of $\mathcal{F}(\Omega)$:
\begin{align*}
\mathcal{F}_j(\Omega):=
\left\{
\mathbf{x}\in \mathcal{F}(\Omega) \mid
\mathbf{x}\text{ is at most $j$ edges away from }
\Omega_\phi(2\cdot Y^{\frac{1}{\delta}})
\right\}.
\end{align*}
Note that by Proposition~\ref{thm:bow}, $\Omega_\phi(2\cdot Y^{\frac{1}{\delta}})$ is always non-empty and hence $\mathcal{F}_j(\Omega)$ is always well-defined. The $\{\mathcal{F}_j(\Omega)\}$ give a filtration of $\mathcal{F}(\Omega)$, and we prove Lemma~\ref{thm:technicalcompare} by showing that equation~\eqref{eq:bound} holds for every $\mathcal{F}_j(\Omega)$. In particular, we do this by induction with respect to $j$.\medskip

\textbf{Base case $(j=1)$}: our aim is establish the base case tautologically by choosing $M_k$ so that every arc corresponding to the complementary regions in $\mathcal{F}_0(\Omega)=\Omega_{\phi}(2\cdot Y^{\frac{1}{\delta}})$ satisfies equation~\eqref{eq:bound}. By definition,  $\mathcal{F}_1(\Omega)$ is equal to the union of all the complementary regions within one edge of the complementary regions constituting $\Omega_\phi(2\cdot Y^{\frac{1}{\delta}})$. Since $\Omega_\phi(2\cdot Y^{\frac{1}{\delta}})$ consists of finitely many cells, it suffices to construct $M_k$ for the one-edge neighborhood of an arbitrary complementary region $\mathbf{a}$ consisting of $\{\mathbf{a},\mathbf{b}_i,\mathbf{c}_i\}_{i\in\mathbb{Z}}$. We also impose, for our later convenience, the condition that $M_k\geq \epsilon(h) M_{k-1}$ and 
\begin{align}
M_k\geq M_1\cdot M_{k-1}+M_2\cdot M_{k-2}+\ldots+ M_{k-1}\cdot M_1.
\end{align}
for $k=1,\ldots,\lfloor{N/2}\rfloor$.\medskip

We know from Lemma~\ref{thm:spiralbound} that $||s_{2k}(b_i)||=O(|i|^kR^{|i|})$ as $|i|$ tends to infinity and is asymptotically dominated by $R^{|i|(1+k\delta)}$ since $R>1$. Similarly, we know that $s_{2k}(c_i)=O(|i|^{2k}R^{2|i|})$ as $|i|$ tends to infinity and is asymptotically dominated by $R^{2|i|(1+k\delta)}$. This suffices to let us construct appropriate $M$-terms.\medskip

\textbf{Induction step}: consider an arbitrary complementary region $\textbf{a}\in \mathcal{F}_{j+1}(\Omega)-\mathcal{F}_j(\Omega)$ configured so that the edge between $\mathbf{a}$ and $\mathcal{F}_j(\Omega)$ borders $\mathbf{b}$ and $\mathbf{c}$. We assume without loss of generality that $\mathbf{c}$ is strictly closer to $\Omega_\phi(2\cdot Y^{\frac{1}{\delta}})$ than $\mathbf{b}$ (see figure~\ref{fig:proof}). Using Remark~\ref{rmk:nongeneric}, we derive the following lemmas.

\begin{lemma}
We have $\frac{\epsilon(b)}{\epsilon(c)\epsilon(d)}\geq\frac{\epsilon(h)}{2}$. 
\end{lemma}
\begin{proof}
Proposition~\ref{thm:bow} tells us that the sink vertex $v_0$ must lie in $\Omega(2\cdot Y^{\frac{1}{\delta}})$, and hence $\mathbf{c}$ is strictly closer to $v_0$ than $\mathbf{b}$ without either region being adjacent to $v_0$. This in turn implies that the edge shared by $\mathbf{c}$ and $\mathbf{d}$ points away from $\mathbf{b}$, thus giving the desired inequality. 
\end{proof}
\begin{lemma}
We have $\epsilon(a^*)=\max\{\epsilon(b),\epsilon(c)\}$. 
\end{lemma}
\begin{proof}The vertex $v_\mathbf{a}$ shared by $\mathbf{a},\mathbf{b},\mathbf{c}$ is the nearest vertex from $\mathbf{a}$ to the sink vertex $v_0$, and hence all of the edges bounding $\mathbf{a}$ point toward $v_\mathbf{a}$. Denoting the sequence of complementary regions encircling $\mathbf{a}$ by  $\{\mathbf{b}_i\}$ with $\mathbf{b}_1=\mathbf{b}$ and $\mathbf{b}_0=\mathbf{c}$, it follows that $\epsilon(b)$ is the minimum among all $\{\epsilon(b_{2i+1})\}$ with odd subscripts and that $\epsilon(c)$ is the minimum among all $\{\epsilon(b_{2i})\}$ with even subscripts. Since $\epsilon(c)\leq\epsilon(b)$, this means that $\epsilon(c)$ is the minimum among the sequence $\{\epsilon(b_i)\}$. In fact, it is easy to see from the algebraic structure of the body part of \eqref{eq:recursol} that the second smallest $\epsilon(b_i)$ must be in $\{\epsilon(b_{1}),\epsilon(b_{-1})\} $ and hence must be $\epsilon(b)\leq\epsilon(b_{-1})$.
\end{proof}

\begin{lemma}
We have 
$$\max\{\epsilon(b),\epsilon(c)\}\geq\epsilon(d^*) ~{\rm and }~\max\{\epsilon(c),\epsilon(d)\}\geq \epsilon(b^*).$$
\end{lemma}
\begin{proof}
These two inequalities hold because $\mathbf{b},\mathbf{c},\mathbf{d}$ all meet at a common vertex.
\end{proof}

\begin{figure}[h!]
\includegraphics[scale=1]{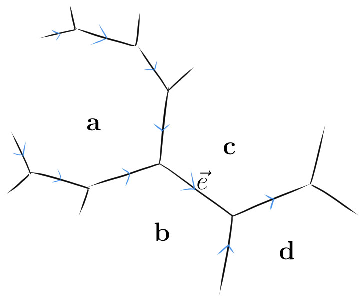}
\caption{The directed edge $\vec{e}$ borders $\mathbf{b},\mathbf{c}$ and points from $\mathbf{a}$ to $\mathbf{d}$.}
\label{fig:proof}
\end{figure}

Proceeding by induction on $j$, we now make the assumption that equation~\eqref{eq:bound} is satisfied by every element of $\mathcal{F}_j(\Omega)$. Consider the Ptolemy relation $ad=b^2+c^2+bcW_a$ restricted to degree $2k$ terms:
\begin{align*}
s_{2k}(a)s_0(d)
&=\sum_{j=0}^{k}\left(s_{2j}(b)s_{2k-2j}(b)+s_{2j}(c)s_{2k-2j}(c)\right)\\
&+W_a\sum_{j=0}^{k-1} s_{2j}(b)s_{2k-2-2j}(c)
-\sum_{j=0}^{k-1}s_{2j}(a)s_{2k-2j}(d).
\end{align*}

We now sequentially verify equation~\eqref{eq:bound} for $k=0,\ldots, [N/2]$ in what is effectively a nested induction. First observe that \eqref{eq:bound} is trivially true for $k=0$. For $k\geq 1$, by invoking the following: 
\begin{itemize}
\item
the $\mathcal{F}_j(\Omega)$ induction assumption on $b,c,d$ and
\item
the (nested) induction assumption that \eqref{eq:bound} holds for lower degree terms in $a\in\mathbb{R}_{s[N]}$,
\end{itemize}
we obtain
\begin{align*}
\epsilon(d)\|s_{2k}(a)\|
\leq~&(\epsilon(b)^2\epsilon(b^*)^{k\delta}+\epsilon(c)^2\epsilon(c^*)^{k\delta})
\sum_{j=0}^k M_j\cdot M_{k-j}\\
&+\sum_{j=0}^{k-1}M_j\cdot M_{k-1-j}
\cdot
\epsilon(b)\epsilon(c)\epsilon(b^*)^{j\delta}\epsilon(c^*)^{(k-1-j)\delta}\\
&+\sum_{j=0}^{k-1}M_{j}\cdot M_{k-j}\cdot
\epsilon(a)\epsilon(d)\epsilon(a^*)^{j\delta}\epsilon(d^*)^{(k-j)\delta}\\
\leq~&3M_k\cdot (\epsilon(b)^2+\epsilon(c)^2)
\cdot(\max\{\epsilon(c),\epsilon(d)\})^{k\delta}\\
&+3M_{k-1}\cdot\tfrac{1}{2}(\epsilon(b)^2+\epsilon(c)^2)
\cdot (\max\{\epsilon(c),\epsilon(d)\})^{(k-1)\delta}\\
&+2M_k \cdot \epsilon(a)\epsilon(d)\epsilon(a^*)^{(k-1)\delta}\epsilon(d^*)^\delta\\
\leq~&\epsilon(a)\epsilon(d)
\left(
3M_k\cdot\max\{\epsilon(c),\epsilon(d)\}^{k\delta}\right.\\
&\left.+\tfrac{3}{2}M_{k-1}\cdot \max\{\epsilon(c),\epsilon(d)\}^{(k-1)\delta}
+2M_k \cdot \epsilon(a^*)^{(k-1)\delta}\epsilon(d^*)^\delta\right).
\end{align*}

The classical Fuchsian case asserts that $\epsilon(h)\epsilon(\max\{\epsilon(c),\epsilon(d)\})>1$ and hence
\[
\tfrac{3}{2}M_{k-1}\cdot \max\{\epsilon(c),\epsilon(d)\}^{(k-1)\delta}<\tfrac{3}{2}\epsilon(h)^{\delta}M_{k-1}\cdot \max\{\epsilon(c),\epsilon(d)\}^{k\delta}
\]
Moreover, since $M_k\geq \epsilon(h) M_{k-1}$, we have 
\begin{align}
\|s_{2k}(a)\|
\leq M_k\cdot\epsilon(a)
\left(
(3+\tfrac{3}{2}\epsilon(h)^{\delta-1})(\max\{\epsilon(c),\epsilon(d)\})^{k\delta}\right.&\notag\\
\left.
+2\epsilon(a^*)^{(k-1)\delta}\epsilon(d^*)^\delta
\right).&\label{eq:intermediate}
\end{align}
We also have
\begin{align}
\epsilon(d^*)
&\leq\max\{\epsilon(c),\epsilon(d)\}\notag\\
&=\frac{\epsilon(c)\epsilon(d)}{\min\{\epsilon(c),\epsilon(d)\}}
=\frac{\epsilon(b)}
{\frac{\epsilon(b)}{\epsilon(c)\epsilon(d)}\cdot \min\{\epsilon(c),\epsilon(d)\}}\notag\\
&\leq\frac{\epsilon(a^*)}
{\frac{\epsilon(h)}{2}\cdot \min\{\epsilon(c),\epsilon(d)\}}\leq Y^{-\frac{1}{\delta}}\cdot\epsilon(a^*),\label{eq:ineq}
\end{align}
where the last inequality uses the fact that $\mathbf{c},\mathbf{d}\notin\Omega_\phi(2\cdot Y^{\frac{1}{\delta}})$. Substituting \eqref{eq:ineq} into \eqref{eq:intermediate} then yields
\begin{align*}
\|s_{2k}(a)\|
\leq M_k\cdot\epsilon(a)\cdot \epsilon(a^*)^{k\delta}((3+\tfrac{3}{2}\epsilon(h)^{\delta-1})\cdot Y^{-k}+2\cdot Y^{-1}).
\end{align*}
The only condition we have hitherto imposed on $Y$ is that $Y^{\frac{1}{\delta}}\geq\tfrac{3}{2}$. By further setting $Y$ large enough so that $(3+\tfrac{3}{2}\epsilon(h)^{\delta-1})\cdot Y^{-k}+2\cdot Y^{-1}<1$, we see that
\[
\|s_{2k}(a)\|
<M_k\cdot\epsilon(a)\cdot\epsilon(a^*)^{k\delta},
\]
as desired.
\end{proof}

Theorem~\ref{thm:comparison} is a simple consequence of Lemma~\ref{thm:technicalcompare}:

\begin{proof}[Proof of theorem~\ref{thm:comparison}]
Take $M=\max\{M_0,\ldots, M_{\lfloor\frac{N}{2}\rfloor}\}$. Observe that for any complementary region $\mathbf{a}\in F(\Omega)- \Omega_\phi(2\cdot Y^{\frac{1}{\delta}})$, we have $\epsilon(a)\geq\epsilon(a^*)$, and the desired comparison holds. There are finitely many remaining cases, and we increase $M$ to encompass these if necessary.
\end{proof}

To conclude this section, we apply the Body-soul Comparison Theorem \ref{thm:comparison} to derive the asymptotic growth rate of the simple super $\lambda$-length spectrum.

\begin{corollary}[Asymptotic growth rate]
Given a monodromy representation $\rho$ for a marked once-punctured super torus $[\rho]\in ST(F)$, we define
\begin{align}
N_\rho(L):=
\#\left\{
\mathbf{a}\in F(\Omega):~
\log\|a\|<L
\right\}.
\end{align}
Then the following limit
\[
\eta(\rho):=
\lim_{L\to\infty}
\frac{N_\rho(L)}{L^2}
\]
exists and varies properly and continuously with respect to the underlying Fuchsian representation $\epsilon(\rho)$ for $\rho$.
\end{corollary}

\begin{proof}
Since $\|a\|>\epsilon(a)$,  $N_\rho(L)$ must be less than or equal to its classical analogue $N_{\epsilon(\rho)}(L)$. On the other hand, the Body-soul Comparison Theorem 
\ref{thm:comparison} tells us that $N_\rho(L)\geq N_{\epsilon(\rho)}(\frac{L-\log M_\delta}{1+\delta})$, and we thus have the following sandwich
\begin{align*}
\lim_{L\to\infty}\frac{N_{\epsilon(\rho)}(L)}{L^2}
\geq\lim_{L\to\infty}\frac{N_\rho(L)}{L^2}
&\geq\lim_{L\to\infty}\frac{N_{\epsilon(\rho)}(\frac{L-\log M_\delta}{1+\delta})}{L^2}\\
&=\frac{1}{(1+\delta)^2}
\lim_{L\to\infty}\frac{N_{\epsilon(\rho)}(L)}{L^2}.
\end{align*}
Since $\delta>0$ may be taken to be arbitrarily small, the (dominant term in the) asymptotic growth rate of $N_\rho(L)$ is equivalent to that of $N_{\epsilon(\rho)}(L)$. The asymptotic behavior of $N_{\epsilon{\rho}}(L)$ is well-understood: it is the asymptotic behavior of the number of simple closed geodesics shorter than $L$ on the hyperbolic surface corresponding to $\epsilon(\rho)$ (see \cite{Mgro}). McShane-Rivin establish quadratic growth rate in \cite{MR} by relying on Zagier's analysis of Markoff triples \cite {zagier}.
\end{proof}

\section{Proof of the super McShane identity}
\label{sec:proof}

The main goal of this section is to prove 

\begin{theorem}
\label{thm:supermcshane}
Fix a once-punctured super torus $F$ and let $\mathcal{S}(F)$ denote the set of free homotopy classes of simple closed curves on $F$.  Then
\begin{align}
\frac{1}{2}
=\sum_{\gamma\in \mathcal{S}(F)}
\left(
\frac{1}{e^{\ell_\gamma}+1}
+\frac{W_\gamma}{4}~\frac{\sinh\frac{\ell_\gamma}{2}}
{\cosh^2\frac{\ell_\gamma}{2}}
\right)\, ,
\end{align}
where $\ell_\gamma$ is the super length of $\gamma$, and $W_\gamma:=W_c$ for the unique ideal arc $c$ disjoint from $\gamma$.
\end{theorem}

\subsection{Absolute convergence}

We begin by showing that the super McShane identity expresses an absolutely convergent series. Although this step is not, strictly speaking, necessary for establishing the identity, it is nevertheless convenient for justifying any combinatorial manipulations we wish to conduct with the summation indices for our identity. We employ the following useful facts

\begin{lemma}\label{thm:usefulbounds}
For each $\mathbf{a}\in F(\Omega)$, let $a=\phi(\mathbf{a})$ denote its corresponding $\lambda$-length.  Then we have
\begin{itemize}
\item
for each $\delta>0$, there exists a constant $C_\delta$ such that for every $\mathbf{a}\in F(\Omega)$
\begin{align}
\left\|\tfrac{1}{a}\right\|< \tfrac{C_\delta}{\epsilon(a)^{1-\delta}}\,;\label{eq:reciprocal}
\end{align}
\item
there exists a constant $C>0$ such that for every $\mathbf{a}\in F(\Omega)$
\begin{align}
\left\|
\left(
\sqrt{1-\tfrac{4}{a^2h^2}}
\right)^{-1}
\right\|
\text{ and }
\left\|
\left(
1+\sqrt{1-\tfrac{4}{a^2h^2}}
\right)^{-1}
\right\|
<C.\label{eq:sqrtbound}
\end{align}
\end{itemize}
\end{lemma}

\begin{proof}
The first inequality is a consequence of Theorem~\ref{thm:comparison}. To see this, observe that
\begin{align*}
\frac{1}{a}
=\frac{1}{\epsilon(a)}
\cdot
\frac{1}{1+\frac{s(a)}{\epsilon(a)}}
=\frac{1}{\epsilon(a)}\cdot
\left(
1-\tfrac{s(a)}{\epsilon(a)}+\ldots
+(-1)^{\lfloor N\rfloor}
\left(\tfrac{s(a)}{\epsilon(a)}\right)^{\lfloor N\rfloor}
\right).
\end{align*}
Let $M$ be a bounding constant satisfying Theorem~\ref{thm:comparison} for $\frac{\delta}{\lfloor N\rfloor}$, so that
\begin{align*}
\left\|
\frac{1}{a}
\right\|
&\leq
\frac{1}{\epsilon(a)}
+\left\|\frac{s(a)}{\epsilon(a)^2}\right\|
+\ldots
+\left\|\frac{s(a)^{\lfloor N\rfloor}}{\epsilon(a)^{\lfloor N\rfloor+1}}\right\|\\
&\leq
\frac{1}{\epsilon(a)}
+\frac{M}{\epsilon(a)^{1-\frac{\delta}{\lfloor N\rfloor}}}
+\ldots
+\frac{M^{\lfloor N\rfloor}}{\epsilon(a)^{1-\delta}}.
\end{align*}
The rightmost term in the above sum is dominant and we may therefore find $C_\delta$ satisfying \eqref{eq:reciprocal}, as required.\medskip

We derive the second inequality via a compactness argument and first observe that $\left\|
1-\tfrac{4}{a^2h^2}
\right\|
\leq
1+4\left\|\tfrac{1}{a}\right\|^2\left\|\tfrac{1}{h}\right\|^2$. Thanks to \eqref{eq:reciprocal}, this is bounded above by a constant independent of $\mathbf{a}\in F(\Omega)$. Moreover, $\epsilon(ah)=2\cosh(\frac{\epsilon(\ell_{\gamma})}{2})>2$, and the discreteness (and strict positivity) of the classical $\lambda$-length spectrum ensures that the points $\{1-\frac{4}{a^2h^2}\}_{\mathbf{a}\in F(\Omega)}$ lie within a compact subset of 
\[
\{x\in\mathbb{R}_{S[N]}\mid \epsilon(x)>0\}
\]
with respect to the na\"{i}ve topology on $\mathbb{R}_{S[N]}$. Since $f(x)=\|\frac{1}{\sqrt{x}}\|$ and $g(x)=\|\frac{1}{1+\sqrt{x}}\|$ are well-defined and continuous (for the na\"{i}ve topology) functions on $\{x\in\mathbb{R}_{S[N]}\mid \epsilon(x)>0\}$, the aforementioned compactness then ensures the existence of the requisite upper bound in \eqref{eq:sqrtbound}.

\end{proof}

\begin{proposition}\label{thm:absconvergence}
Given a once-punctured super torus $F$, the series
\begin{align}
\sum_{\mathbf{a}\in F(\Omega)}
\left\|
\frac{1}{ahr_{\mathbf{a}}}
+\frac{W_a}{2ah}
\right\|
\end{align}
 is absolutely convergent.
\end{proposition}

\begin{proof}
We make the following comparisons
\begin{align}
\left\|
\frac{1}{ahr_{\mathbf{a}}}
+\frac{W_a}{2ah}
\right\|
&\leq
\left\|
\frac{1}{ahr_{\mathbf{a}}}
\right\|
+
\left\|\frac{W_a}{2ah}
\right\|
\notag\\
&\leq
\left\|
\frac{1}{a}
\right\|^2
\left\|
\frac{1}{h}
\right\|^2
\left\|
\left(1+\sqrt{1-\tfrac{4}{a^2h^2}}\right)^{-1}
\right\|
+
\frac{1}{2}
\left\|\frac{1}{a}
\right\|
\left\|\frac{1}{h}
\right\|\nonumber
\end{align}
Fix some small $\delta$ such as $\delta=\frac{1}{2}$, Lemma~\ref{thm:usefulbounds} ensures that
\begin{align}
\sum_{\mathbf{a}\in F(\Omega)}
\left\|
\frac{1}{ahr_{\mathbf{a}}}
+\frac{W_a}{2ah}
\right\|
<
\sum_{\mathbf{a}\in F(\Omega)}
\frac{C_{\frac{1}{2}}}{\epsilon(a)^{\frac{1}{2}}}\label{eq:fincompare}
\end{align}
for some constant $C_{\frac{1}{2}}>0$. Finally, observe that
\begin{align*}
\epsilon(a)^{\frac{1}{2}}
=e^{\frac{1}{2}\log\epsilon(ah)}/\epsilon(h)^{\frac{1}{2}}
>\frac{(\frac{1}{2}\log\epsilon(ah))^3}
{3!\cdot\epsilon(h)^{\frac{1}{2}}}
=C(\log\epsilon(ah))^3>0.
\end{align*}
Since $\epsilon(ah)=2\cosh(\epsilon(\ell_{\gamma_a})/2)$, the spectrum of $\{\log\epsilon(ah)\}_{\mathbf{a}\in\Omega}$ asymptotically approaches the spectrum of simple geodesic lengths on the body of the once-punctured super torus. Therefore, we only need to show that $\sum_\gamma \epsilon(\ell_\gamma)^{-3}$ converges to show that the righthand side of \eqref{eq:fincompare} converges absolutely. From \cite{MR}, we know that the asymptotic growth rate of the number of simple closed geodesics of length less than $L$ is $CL^2+O(L\log L)$. Therefore, 
\begin{align*}
\sum_\gamma \epsilon(\ell_\gamma)^{-3}
&\approx
\sum_{L=1}^\infty
\frac{\#\left\{\text{number of $\gamma$ of length in $(L-1,L]$} \right\}}{L^3}\\
&\approx
\sum_{L=1}^\infty
\frac{CL^2-C(L-1)^2+O(L\log L)}{L^3}
<\text{Constant}\times 
\sum_{L=1}^\infty\frac{1}{L^{1.5}}<\infty.
\end{align*}
\end{proof}

\subsection{Proof of the super McShane identity}

As in the proof of Proposition~\ref{thm:absconvergence}, we first establish a few useful bounds.

\begin{lemma}
For any pair of adjacent $\mathbf{a},\mathbf{b}\in F(\Omega)$ and their corresponding $\lambda$-lengths $a=\phi(\mathbf{a})$ and $b=\phi(\mathbf{b})$, there exists a constant $C>0$ such that
\begin{align}
\left\|
\left(
\sqrt{1-\tfrac{4}{a^2h^2}}
+\sqrt{1-\tfrac{4}{a^2h^2}-\tfrac{4}{b^2h^2}}
\right)^{-1}
\right\|<&C\\
\text{ and }
\left\|
\left(
\sqrt{1-\tfrac{4}{a^2h^2}-\tfrac{4}{b^2h^2}}
\right)^{-1}
\right\|<&C.
\end{align}
\end{lemma}

Strictly speaking, the expression $\sqrt{1-\tfrac{4}{a^2h^2}-\tfrac{4}{b^2h^2}}$ is ill-defined if 
\[
\epsilon(1-\tfrac{4}{a^2h^2}-\tfrac{4}{b^2h^2})=0. 
\]
This only occurs in the non-generic case (Remark~\ref{rmk:nongeneric}), and even then, it occurs for at most one pair of $\mathbf{a},\mathbf{b}$. Our statement above therefore excludes this pair in the non-generic case.

\begin{proof}
The proof for these comparisons follows from the same compactness argument as the proof of \eqref{eq:sqrtbound}. We require the the following observations
\begin{itemize}
\item
$\|1-\tfrac{4}{a^2h^2}-\tfrac{4}{b^2h^2}\|$ is bounded above;
\item
 $\epsilon(1-\tfrac{4}{a^2h^2}-\tfrac{4}{b^2h^2})> 0$ as follows from the fact that
\begin{align}
\epsilon(\psi(\vec{e}))
=\epsilon\left(\frac{1}{2}
\left(
1-\sqrt{1-\tfrac{4}{a^2h^2}-\tfrac{4}{b^2h^2}}
\right)\right)
\end{align}
for the oriented edge $\vec{e}$ wedged between $\mathbf{a}$ and $\mathbf{b}$, such that the head of $\vec{e}$ is closer to the sink $v_0$ than the tail.
\end{itemize}
The remainder of the proof proceeds as before.
\end{proof}

\begin{theorem}\label{thm:final}
For any once-puntured super torus $F$, we have
\begin{align}
\sum_{\mathbf{a}\in F(\Omega)}
\left(\frac{1}{ahr_{\mathbf{a}}}
+\frac{W_a}{2ah}\right)
=\frac{1}{2},
\end{align}
where $a=\phi(\mathbf{a})$ and $r_{\mathbf{a}}=\tfrac{1}{2}(ah-W_a+\sqrt{(ah-W_a)^2-4})$.
\end{theorem}

\begin{proof}
Let $v_0\in V(\Omega)$ denote the sink vertex, and let $T_n$ denote the finite tree whose edges consist of all the edges $E(\Omega)$ within $n$ edges from $v$. For each $T_n$, Proposition~\ref{thm:finitesum} confirms that
\begin{align}
\sum_{\vec{e}\in C(T_n)}\psi(\vec{e})=1.\label{eq:finitesum}
\end{align}
Our goal is to show that the above sum becomes a good approximation to the putative super McShane identity as $n$ tends to infinity. We first observe that there is a natural map $\iota_n:C(T_n)\to F(\Omega)$ which sends an oriented edge $\vec{e}\in C(T_n)$ to the complementary region bordered by $\vec{e}$ which is closer to $v_0$. For example in Figure~\ref{fig:proof}, the edge $\vec{e}$ maps to $\mathbf{c}$ as it is closer to $v_0$ than $\mathbf{b}$.\medskip

Consider an arbitrary complementary region $\mathbf{a}\in F(\Omega)$ and let $v_{\mathbf{a}}$ be the vertex on the boundary of $\mathbf{a}$ closest to $v_0$. If $T_n$ contains $v_0$ then there must be at least two edges in $C(T_n)$ lying on the boundary of $\mathbf{a}$. There can be at most one edge in $C(T_n)$ clockwise from $v_{\mathbf{a}}$ as $T_n$ will contain all but the farthest edge from $v_{\mathbf{a}}$; likewise, there can be at most one edge counter-clockwise from $v_{\mathbf{a}}$. It follows that $\iota_n$ is a $2$-to-$1$ map onto its image.  In fact, it is fairly straight-forward to verify that the image of $\iota_n$ is precisely $F(T_{n-1})$.\medskip

We now use $\iota_n$ to convert sums over $C(T_n)$ into sums over $F(T_{n-1})=\iota_n(C(T_n))$ and compare the difference of the corresponding summands. Our aim is to show that the \emph{$n$-th error term series} approaches $0$, that is,
\begin{align}
\lim_{n\to\infty}
\sum_{
\substack{
\mathbf{a}\in F(T_{n-1})\\
\vec{e}\in\iota_n^{-1}(\{\mathbf{a}\})}
}
\left(
\psi(\vec{e})
-\left(\frac{1}{ahr_{\mathbf{a}}}
+\frac{W_a}{2ah}\right)
\right) = 0.\label{eq:limit}
\end{align}

For $\vec{e}$ with head pointing towards $\mathbf{c}$ and bordering $\mathbf{a},\mathbf{b}$ with $\mathbf{a}$ closer to $v_0$ than $\mathbf{b}$, recall that 
\begin{align*}
\psi(\vec{e})
=\tfrac{c}{abh}+\tfrac{W_a}{2ah}+\tfrac{W_b}{2bh}
\end{align*}
and solve for $c$ in terms of $a$ and $b$ using the vertex relation
\[
c=\tfrac{1}{2}\left(abh-aW_b-bW_a\pm
\sqrt{(aW_b+bW_a-abh)^2-4(a^2+b^2+abW_c)}\right).
\]
Since $\epsilon(\frac{c}{ab})<\frac{1}{2}$, we must take the negative sign, and so
\begin{align}
\psi(\vec{e})
=\frac{1}{2}
\left[
1-\sqrt{1-\tfrac{4}{a^2h^2}-\tfrac{4}{b^2h^2}}
+\frac{\frac{W_a}{ah}+\frac{W_b}{bh}+\frac{2W_c}{abh^2}}
{\sqrt{1-\frac{4}{a^2h^2}-\tfrac{4}{b^2h^2}}}
\right].
\end{align}
Recall also that
\begin{align}
\frac{1}{ahr_{\mathbf{a}}}
+\frac{W_a}{2ah}
=\frac{1}{2}
\left[
1-\sqrt{1-\tfrac{4}{a^2h^2}}
+\frac{\frac{W_a}{ah}}
{\sqrt{1-\tfrac{4}{a^2h^2}}}
\right].
\end{align}
Taking the difference of these two equations yields
\begin{align}
\left\|
\psi(\vec{e})
-\left(\tfrac{1}{ahr_{\mathbf{a}}}
+\tfrac{W_a}{2ah}\right)
\right\|
\leq&
\left\|
\sqrt{1-\tfrac{4}{a^2h^2}}
-\sqrt{1-\tfrac{4}{a^2h^2}-\tfrac{4}{b^2h^2}}
\right\|\label{eq:term1}\\
+&
\left\|
\frac{\frac{W_a}{ah}}
{\sqrt{1-\frac{4}{a^2h^2}-\tfrac{4}{b^2h^2}}}
-\frac{\frac{W_a}{ah}}
{\sqrt{1-\tfrac{4}{a^2h^2}}}
\right\|\label{eq:term2}\\
+&
\left\|
\frac{\frac{W_b}{bh}+\frac{2W_c}{abh^2}}
{\sqrt{1-\frac{4}{a^2h^2}-\tfrac{4}{b^2h^2}}}
\right\|.\label{eq:term3}
\end{align}
The righthand term in equation~\eqref{eq:term1} in turn satisfies
\begin{align}
&\left\|
\sqrt{1-\tfrac{4}{a^2h^2}}
-\sqrt{1-\tfrac{4}{a^2h^2}-\tfrac{4}{b^2h^2}}
\right\|=
\left\|
\frac{\frac{4}{b^2h^2}}
{\sqrt{1-\tfrac{4}{a^2h^2}}
+\sqrt{1-\tfrac{4}{a^2h^2}-\tfrac{4}{b^2h^2}}}
\right\|\notag\\
\leq&
\left\|
\frac{4}{b^2h^2}
\right\|
\cdot
\left\|
\left({\sqrt{1-\tfrac{4}{a^2h^2}}
+\sqrt{1-\tfrac{4}{a^2h^2}-\tfrac{4}{b^2h^2}}}
\right)^{-1}
\right\|<\frac{C\cdot C_{\frac{\delta}{2}}}{\epsilon(b)^{2-\delta}}\,.
\end{align}
For the term in \eqref{eq:term2}, we have
\begin{align}
&\left\|
\frac{\frac{W_a}{ah}}
{\sqrt{1-\frac{4}{a^2h^2}-\tfrac{4}{c^2h^2}}}
-\frac{\frac{W_a}{ah}}
{\sqrt{1-\tfrac{4}{a^2h^2}}}
\right\|\notag\\
=&
\left\|
\frac{\frac{W_a}{ah}\left(\sqrt{1-\tfrac{4}{a^2h^2}}
-\sqrt{1-\tfrac{4}{a^2h^2}-\tfrac{4}{b^2h^2}}\right)}
{{\sqrt{1-\frac{4}{a^2h^2}-\tfrac{4}{b^2h^2}}}\cdot {\sqrt{1-\tfrac{4}{a^2h^2}}}}
\right\|<\frac{C\cdot C_{\delta}\cdot C_{\frac{\delta}{2}}}{\epsilon(a)^{1-\delta}\epsilon(b)^{2-\delta}}\,,
\end{align}
and for the term in \eqref{eq:term3} have
\begin{align}
&\left\|
\frac{\frac{W_b}{bh}+\frac{2W_c}{abh^2}}
{\sqrt{1-\frac{4}{a^2h^2}-\tfrac{4}{b^2h^2}}}
\right\|\notag\\
&\leq
\left\|
\frac{1}{bh}
\right\|
\cdot
\left(1+\left\|
\frac{2}{ah}
\right\|\right)
\cdot
\left\|
\left(
\sqrt{1-\tfrac{4}{a^2h^2}-\tfrac{4}{b^2h^2}}
\right)^{-1}
\right\|
<\frac{C\cdot C_{\delta}}{\epsilon(b)^{1-\delta}}\,.
\end{align}
Fixing $\delta=\frac{1}{2}$, it follows that each summand in the $n$-th error term series is bounded above by $C\epsilon(b)^{-\frac{1}{2}}$, for some constant $C>0$, where 
\[
\mathbf{b}\in F(T_n)-F(T_{n-1})\subset F(\Omega)-F(T_{n-1}).
\]
We may thus conclude
\begin{align*}
\left\|\sum_{
\substack{
\mathbf{a}\in F(T_{n-1})\\
\vec{e}\in\iota_n^{-1}(\{\mathbf{a}\})}
}
\left(
\psi(\vec{e})
-\left(\frac{1}{ahr_{\mathbf{a}}}
+\frac{W_a}{2ah}\right)
\right) 
\right\|
<\sum_{\mathbf{b}\in F(\Omega)-F(T_{n-1})}C\,
\epsilon(b)^{-\frac{1}{2}}.
\end{align*}
Since $\sum_{\mathbf{b}\in F(T_{n-1})}
\epsilon(b)^{-\frac{1}{2}}$ converges and $\cup_{n} F(T_n)=F(\Omega)$, equation~\eqref{eq:limit} must hold and hence
\begin{align}
\lim_{n\to\infty}
\sum_{
\substack{
\mathbf{a}\in F(T_{n-1})\\
\vec{e}\in\iota_n^{-1}(\{\mathbf{a}\})}
}
\psi(\vec{e})
=&
\lim_{n\to\infty}
\sum_{
\substack{
\mathbf{a}\in F(T_{n-1})\\
\vec{e}\in\iota_n^{-1}(\{\mathbf{a}\})}
}
\left(\frac{1}{ahr_{\mathbf{a}}}
+\frac{W_a}{2ah}\right),\\
{\rm which~gives}~~~~~~
\lim_{n\to\infty}
\sum_{
\vec{e}\in C(T_n)
}
\psi(\vec{e})
=&
\lim_{n\to\infty}
\sum_{\mathbf{a}\in F(T_{n-1})}
2\left(\frac{1}{ahr_{\mathbf{a}}}
+\frac{W_a}{2ah}\right)\\
{\rm whence}~
1=&
\sum_{\mathbf{a}\in F(\Omega)}
2\left(\frac{1}{ahr_{\mathbf{a}}}
+\frac{W_a}{2ah}\right),
\end{align}
where the last line here is obtained from substituting in \eqref{eq:finitesum}.
\end{proof}

\begin{proof}[Proof of Theorem~\ref{thm:supermcshane}]
The super McShane identity (Theorem~\ref{thm:supermcshane}) is equivalent to Theorem~\ref{thm:final} after substuting in the following formulae
\begin{align*}
ah=r_{\mathbf{a}}+r_{\mathbf{a}}^{-1}+W_a\text{ and }r_{\mathbf{a}}=e^{\frac{\ell_\alpha}{2}},
\end{align*}
where $\mathbf{a}$ denotes the complementary region corresponding to the simple closed curve $\alpha$.
\end{proof}

\end{document}